               \documentclass[11pt,letterpaper]{amsart}

\usepackage{amsmath}
\usepackage{amsfonts}
\usepackage{amssymb}
\usepackage{amsthm}
\usepackage{url}
\usepackage{color}
\usepackage{array}
\usepackage{array}
\usepackage{color}
\usepackage{verbatim}
\usepackage{cancel}
\usepackage[utf8x]{inputenc}
\usepackage{t1enc}
\usepackage[english]{babel}
\usepackage{hyperref}

\usepackage{amsmath,amssymb,amsthm}
\usepackage{mathrsfs,esint}
\usepackage{graphicx,wrapfig}
\usepackage[format=hang]{caption}
\usepackage{enumitem}

\newcommand{\R}{\mathbb{R}}

\newcommand{\N}{\mathbb{N}}

\DeclareMathOperator*{\elimsup}{ess\,lim\,sup}
\DeclareMathOperator*{\diam}{diam}

\DeclareMathOperator{\Mod}{Mod}

\DeclareMathOperator*{\rad}{rad}

\def\vint_#1{\mathchoice%
	{\mathop{\kern 0.2em\vrule width 0.6em height 0.69678ex depth -0.58065ex
			\kern -0.8em \intop}\nolimits_{\kern -0.4em#1}}%
	{\mathop{\kern 0.1em\vrule width 0.5em height 0.69678ex depth -0.60387ex
			\kern -0.6em \intop}\nolimits_{#1}}%
	{\mathop{\kern 0.1em\vrule width 0.5em height 0.69678ex depth -0.60387ex
			\kern -0.6em \intop}\nolimits_{#1}}%
	{\mathop{\kern 0.1em\vrule width 0.5em height 0.69678ex depth -0.60387ex
			\kern -0.6em \intop}\nolimits_{#1}}}

\def\qed{\enspace\null\hfill $\square$\par\medbreak} 

\newcommand{\Om}{\Omega}

\newcommand{\eps}{\varepsilon}
\newcommand{\pip}{\varphi}

\newcommand{\ch}{\text{\raise 1.3pt \hbox{$\chi$}\kern-0.2pt}}

\newcommand*{\calE}{\mathcal{E}}

\providecommand*{\vint}[1]{\mathchoice
          {\mathop{\vrule width 5pt height 3 pt depth -2.5pt
                  \kern -9pt \kern 1pt\intop}\nolimits_{\kern -5pt{#1}}}
          {\mathop{\vrule width 5pt height 3 pt depth -2.6pt
                  \kern -6pt \intop}\nolimits_{\kern -3pt{#1}}}
          {\mathop{\vrule width 5pt height 3 pt depth -2.6pt
                  \kern -6pt \intop}\nolimits_{\kern -3pt{#1}}}
          {\mathop{\vrule width 5pt height 3 pt depth -2.6pt
                  \kern -6pt \intop}\nolimits_{\kern -3pt{#1}}}}
\newcommand*{\jint}{\fint}

\theoremstyle{plain}
\newtheorem{theorem}[equation]{Theorem}
\newtheorem{lemma}[equation]{Lemma}
\newtheorem{corollary}[equation]{Corollary}
\newtheorem{proposition}[equation]{Proposition}

\numberwithin{equation}{section}

\theoremstyle{definition}
\newtheorem{definition}[equation]{Definition}

\newtheorem{example}[equation]{Example}

\theoremstyle{remark}
\newtheorem{remark}[equation]{Remark}

\makeatletter
\def\l@subsection{\@tocline{2}{0pt}{2.5pc}{5pc}{}}
\makeatother

\begin{document}

 	\title[ Fractional $p$-Laplacian via Neumann problems ]{ Fractional $p$-Laplacians via Neumann problems in unbounded 
	metric measure spaces}
	
	\author[Capogna, Gibara, Korte, Shanmugalingam]{Luca Capogna, Ryan Gibara, Riikka Korte, Nageswari Shanmugalingam}

	
	
	\begin{abstract}
We prove well-posedness, Harnack inequality and sharp regularity of solutions to a fractional $p$-Laplace non-homogeneous 
equation $(-\Delta_p)^su =f$, with $0<s<1$, $1<p<\infty$, 
for data $f$ satisfying
a weighted $L^{p'}$ condition
in a  doubling  metric measure space $(Z,d_Z,\nu)$ that is possibly unbounded. Our approach is inspired by the work of 
Caffarelli and Silvestre \cite{CS} (see also Mol{\v{c}}anov and Ostrovski{\u{i}}~\cite{MO}), 
and  extends the techniques developed in 
 \cite{CKKSS}, where the bounded case is studied. 
Unlike in~\cite{EbGKSS}, we do not assume that $Z$ supports a Poincar\'e inequality.
The proof is based on the well-posedness of the Neumann problem in a uniform metric measure space $(X,d_X, \mu)$ that 
arises as an hyperbolic filling of  $Z$. To the best of our knowledge, the results appear to be new even in the case of an unbounded Euclidean domain.
	\end{abstract}
	
	\date{\today}
	
	\maketitle
	
	\tableofcontents

	\section{Introduction}	

 The recent paper \cite{CKKSS}
 introduced a sweeping generalization 
 of the method of Caffarelli and Silvestre \cite{CS} (see also \cite{MO}) to construct and study fractional powers of the $p$-Laplacian, 
  when $1<p<\infty$, in a bounded doubling metric measure space $(Z, d_Z, \nu)$. 
 The purpose of the present note  is twofold: {\bf (1)} While the pioneering results in \cite{CS, FF, CG}  allow to study nonlocal 
 operators in various model spaces (Euclidean space, Carnot groups or the hyperbolic space), our previous work in \cite{CKKSS} was limited to 
 compact spaces. Here we extend the results of \cite{CKKSS} to the case where the space $Z$ can be
unbounded. {\bf (2)}  
There is a rich, very broad, literature on the fractional $p$-Laplacian in the Euclidean setting, which is not based on the 
Caffarelli-Silvestre paradigm of studying non-local operators as traces of local operators, but rather on the minimization of a 
Besov energy. In Section~\ref{S:examples} we compare  this point of view, and the corresponding results,  with our approach. Section~\ref{S:examples}
also includes several motivations for our work in this paper.

In moving from the compact space case of~\cite{CKKSS} to the unbounded case of the present paper, the key obstacle is
that we loose information about integrability of the functions of interest; unlike in the case of bounded spaces, measurable
functions in unbounded spaces but with globally $p$-inegrable upper gradients need not themselves be globally $p$-integrable.
This destroys our ability to directly use the results of~\cite{CKKSS} as we no longer can apply global estimates, 
and therefore we had to find an alternate way of establishing 
well-posedness of the problem in a more indirect way via approximation methods for the unbounded space itself. In addition to the considerable technical difficulties arising from the unboundedness of the space, we also introduce a novel class of weighted $L^p$-spaces for the Neumann data which appears to be new even in the Euclidean setting.

  Our results in \cite{CKKSS} and in the present paper are motivated by the larger problem of using nonlocal operators,  like the fractional powers of the $p$-Laplacian,
 on a complete unbounded metric measure space $Z$ that is the  boundary  of a uniform domain $\Omega$ to {\it understand the 
 geometry of $\Omega$ itself,} 
  with such nonlocal operators arising as
 traces of $p$-Laplacian operators on the uniform domains. 
 In analogy to the well known Electrical Impedance Tomography problem, where one wants to recover the scalar coefficient $\gamma$  of an elliptic PDE $\text{div}(\gamma \nabla u)=0$ in a domain $\Omega$ from the study of its Dirichlet-to-Neumann map $u|_{\partial \Omega} \to \partial_\nu u|_{\partial \Omega}$ on the boundary $\partial \Omega$, our  goal here is to develop the study of the trace operators on the boundary
 with the long-term aim of inferring the structure of the uniform domains by understanding the behavior of the trace operator
 on the boundary, as the latter is more visible than the uniform domain itself, which could be  largely hidden from our view in concrete applications.
 
  While in this paper the primary object of interest is the uniform domain  $\Omega$ whose structure is hidden by the structure of
 its boundary $Z=\partial \Omega$,
 there is an alternate viewpoint to that described above, namely, that the primary object of study is the 
 space $Z$  itself rather than the domain it is a boundary of. 
 One of the values of this perspective is in its application to the study of quasiconformal geometry on $Z$, for quasiconformal
 maps preserve certain Besov classes and are characterized by quasipreservation of Besov energy, as demonstrated
 for example in~\cite{LS}.
 From this point of view, one can think of
 the fractional $p$-Laplacian studied in this note as arising in connection to a fractional Besov energy defined 
 in a {\it uniformized hyperbolic filling} of the space $Z$. 
 The idea of hyperbolic filling was first proposed by Gromov \cite{Gr} and implemented in various forms
 in \cite{BSak, BSakSou,BSch,BourPaj,BuSc}. Later  in~\cite{BBS} this notion was extended from the metric setting to the metric measure setting. 
 In conjunction  with the
 tools of uniformization from~\cite{BBS,BHK}, this technique yields that every compact doubling metric measure space
 is (up to a bi-Lipschitz change in the metric) the boundary of a uniform domain which is equipped with a doubling measure supporting
 a $1$-Poincar\'e inequality. 
  Connections between quasiconformal geometry of $Z$ and rough quasiisometry of
 hyperbolic fillings have been explored 
 in~\cite{CG, Car, LS} for instance.
 
 At the heart of the arguments in~\cite{CKKSS} is the idea that Besov spaces, also known as fractional 
 Sobolev spaces, Gagliardo spaces or Slobodetski\u{\i} spaces, are the trace classes of Sobolev spaces of functions on a 
 uniform domain whose boundary is the
 metric measure space of interest. Such an idea is well-known in the Euclidean setting, thanks to the work of  Jonsson and Wallin \cite{JW},
 and  Gagliardo \cite{Gag}. In other words,  in \cite{CKKSS}   
 the bounded space $Z$ is seen as the visual boundary of a hyperbolic filling $(X,d_X, \mu)$ 
 (for instance, one can think of the special example where $Z$ is the unit circle and $X$ is the hyperbolic space in the Poincar\'e model).
 An appropriate uniformization procedure (see \cite{BBS}) yields a John domain $\Omega$, conformally equivalent to $X$, and whose 
 boundary $\partial \Omega$ is bi-Lipschitz equivalent to $Z$. In this context,  the fractional power $(- \Delta_p)^\theta$ of the
 $p$-Laplacian $\Delta_p$ in $Z$ is then realized as the Dirichlet-to-Neumann map for $\Omega$ as described in~\cite{GS3}
 (for more details see Section~\ref{sec:construct-fractLap}). 

 From {these two perspectives,}  in the present paper our  approach is to  
  view an unbounded complete doubling metric measure space $(Z,d_Z,\nu)$ as the 
  boundary of an unbounded uniform
 domain equipped with a measure $(\Om, d,\mu)$ \cite{Bu1}, and
 to study the Dirichlet and Neumann problems for the 
 Cheeger $p$-Laplacian in the metric space $(\Omega, d, \mu)$. We then deduce existence, regularity, 
 and stability of the solutions for the induced fractional powers of the $p$-Laplacian
 on $\partial \Omega$ through the corresponding properties for the Neumann problem. 
 
Correspondingly, in this note we will prove well-posedness of the Neumann problem for the Cheeger $p$-Laplacian 
in unbounded open sets of metric measure spaces satisfying the properties listed below.

\noindent{\bf Structure  hypotheses:}
Throughout the paper, we let $1<p<\infty$ and $\Omega$ be a domain in
a complete metric measure space $(X,d,\mu)$ such that:
\begin{enumerate}\label{structure hypotheses}
\item[(H0)] $\Omega$ is a uniform domain as defined in Subsection~\ref{subsect:Uniform}.
\item[(H1)] $(\overline \Om, d, \mu\vert_{\overline \Om})$ is doubling and supports a $p$-Poincar\'e inequality
as in Subsection~\ref{subsect:doubling}.
\item[(H2)]
The boundary $\partial \Omega$ is  complete, uniformly perfect, and
doubling. Moreover, we assume that it  
is  equipped with a Radon measure $\nu$ for which there are constants $C\ge 1$
and $0<\Theta<p$ such that for all $x\in \partial\Omega$
and $0<r<2\diam(\partial \Omega)$,
\begin{equation}\label{eq:Co-Dim-intro}
 \frac{1}{C}\, \frac{\mu(B(x,r)\cap\Om)}{r^\Theta}\le \nu(B(x,r))\le C\, \frac{\mu(B(x,r)\cap\Om)}{r^\Theta};
\end{equation}
that is, $\nu$ is a $\Theta$-codimensional Hausdorff measure with respect to $\mu\lvert_\Om$.
\end{enumerate}

\begin{remark}
The constants associated with $p$, and the (H0), (H1), and (H2) conditions will be referred to as the structural constants.
Observe also that $\mu(\partial\Om)=0$ because of~\eqref{eq:Co-Dim-intro}.
\end{remark}

We will denote by $D^{1,p}(\Om)$ the homogeneous Sobolev space in $\Om$ and by $HB^\theta_{p,p} (\partial \Omega)$ the 
homogeneous Besov
space in $\partial \Om$. Throughout the paper we will let
\begin{equation}\label{exponent}\theta= 1-\frac{\Theta}{p},\end{equation}
and $\theta\in (0,1)$ will denote the fractional exponent for the fractional $p$-Laplacian operators we work with. 

\bigskip

Next, we turn to the Neumann problem and explain its formulation in our non-smooth context: given data $f\in L^{p'}(\partial \Omega, \nu)$ 
with $\int_{\partial\Om} f\, d\nu=0$,
we say that $u\in D^{1,p}(\Omega)$ solves the Neumann problem 
\[
\begin{cases} \Delta_p u =0 & \text{ in } \Omega \\
|\nabla u|^{p-2} \partial_{\vec n} u=f &\text{ on } \partial \Omega
    \end{cases}
\]
if for all $\phi\in D^{1,p}(\overline{\Omega})$,
      \[
      \int_{\Omega}|\nabla u|^{p-2}\nabla u\cdot \nabla\phi\,d\mu=\int_{\partial\Omega}\phi f\,d\nu
      \]
(see Theorem \ref{thm:equiv-unbounded} for an equivalent definition). In the Euclidean setting, $\nabla u$ is the weak 
derivative of $u$, while in the nonsmooth setting of  doubling metric measure spaces supporting a Poincar\'e inequality it
denotes any choice of a Cheeger differential.

One of our main result is the existence of solutions to the Neumann problem, and the continuity of such solutions with respect to the data.

\begin{theorem}\label{intro-thm:Neumann}
Assume hypotheses (H0), (H1) and (H2) above, and consider data 
\begin{equation}\label{intro-weight-h}
f\in L^{p'}(\partial \Omega,  \nu_J)\cap L^{p'}(\partial\Om,\, \nu),
\end{equation}
with $\int_{\partial \Omega} f d\nu=0$. Here, for a fixed $x_0\in\partial\Omega$, we have set
 \begin{equation}\label{intro-weight}
 J(x,x_0) =  d(x,x_0)^{p^\prime\theta }\, \nu(B(x_0,d(x_0,x)))^{p^\prime/p}
 \end{equation}
 and the measure $\nu_J$ is defined by $\nu_J(E)=\int_E J(y,x_0)d\nu(y)$ 
 whenever $E\subset \partial \Omega$ is a $\nu$-measurable set.
The following properties hold:
\begin{enumerate}
\item { (Existence and uniqueness) } There exists a solution $u\in D^{1,p}(\Om)$ of the 
Neumann problem with data $f$. If $w$ is another solution, then 
$w-u$ is a constant.

\item (Stability with respect to data) 
There exists $C>0$, depending only on the structural constants and on $x_0$, such that for any 
$g\in L^{p'}(\partial \Omega, \nu_J)\cap L^{p'}(\partial\Om,\, \nu)$ with $\int_{\partial \Om} g d\nu=0$, and 
denoting by $v\in D^{1,p}(\Om)$ a solution of the Neumann problem with data $g$, one has the following: If $p\geq 2$, then
\begin{equation}
\|\nabla u-\nabla v\|_{L^p(\Om,\, \mu)}\le C\, \|f-g\|_{L^{p'}(\partial\Om,\, \nu_J)}^{1/(p-1)},
\end{equation}
whereas, if $1<p< 2$, then 
    \begin{multline}
        \| \nabla u - \nabla v\|_{L^p(\Om, \mu)} \\ \le C (\| f\|_{L^{p'}(\partial \Omega, {\nu_J})}
         + \| g\|_{L^{p'}(\partial \Omega, {\nu_J})})^{\frac{2-p}{p-1}} \| f-g\|_{L^{p'}(\partial \Omega, {\nu_J})}
    \end{multline}
\end{enumerate}
\end{theorem}

\begin{remark}
    Note that the choice of the point $x_0$ in \eqref{intro-weight} affects only the constant in the stability estimates.
\end{remark}

 A direct proof that there is a solution to the Neumann problem, as defined above and in 
Subsection~\ref{sub:Def-Neumann},
appears difficult; thus the equivalence between this problem and the equivalent variational interpretation as given in 
Theorem~\ref{thm:equiv-unbounded} below,
is quite useful. Indeed, the existence of solution to the minimization problem can be shown via the direct methods of the calculus of variation.
We overcome the lack of compactness of $\overline{\Om}$ by first considering compactly supported Neumann data, and then
suitably approximating the general Neumann data by compactly supported Neumann data. One of the main technical results 
we use to overcome the apparent lack of compactness is Lemma \ref{panacea-bis}.

\begin{remark}
Although the Euclidean Neumann problem for the  quasilinear PDE in bounded and  unbounded domains has been the subject of several studies, we are not aware of any result in the Euclidean setting that mirrors the well-posedness in Theorem \ref{intro-thm:Neumann}.

Interestingly, the paper~\cite{CNO} considers the Neumann boundary value problem for the Laplace operator ($p=2$ case)
on an unbounded
domain in $\R^2$, with unbounded Lipschitz boundary. Such a boundary is biLipschitz to $\R$, but not necessarily isometric to
$\R$. Should this boundary be isometric to $\R$, we know from~\cite{CS} that the trace of the solution to the Neumann boundary value
problem is the Besov energy minimizer. Without isometry, we do not know that the trace of the solutions obtained in~\cite{CNO}
are Besov energy minimizers, but they are solutions to the trace operator as constructed in the present paper. 
For the more general conformally hyperbolic manifold setting, we refer the 
interested reader to~\cite{CG}.

\end{remark}  
If the data satisfies better integrability conditions, then one can expect H\"older regularity up to the boundary for the solutions.
In the following, we will indicate by $Q_\mu$ the {\it lower mass bound} exponent, defined in the same fashion 
as~\eqref{eq:lower-mass-exp}. Notice that~\eqref{eq:lower-mass-exp} continues to hold even as one  increases arbitrarily the
value of $Q_\mu$ in the lower bound. Hence, if one is willing to pay the price of changing the estimates 
in the H\"older regularity portion of the following
theorem, the assumption $p\le Q_\mu$ should not be considered to be a restrictive one. Also recall that in the range 
$p>Q_\mu$  one already knows that $u$ is $1-Q_\mu/p$-H\"older continuous on
$\overline{\Omega}$ in view of  the
Morrey embedding theorem, see~\cite[Lemma~9.2.12]{HKST} for instance.

\begin{theorem}\label{intro-thm:Neumann-regularity}
Assume hypotheses (H0), (H1) and (H2) above, and let $f$ be as in Theorem \ref{intro-thm:Neumann}.
The following properties hold:
\begin{enumerate}
\item { (Local H\"older regularity at the boundary)} Assume that $1<p\le Q_\mu$. Denote by $B_{R_0}\subset \overline \Omega$  
the intersection of $\overline \Omega$ with a ball of radius $R_0>0$ centered at a point in $\partial \Omega$. If the  boundary data 
satisfies the additional integrability assumption $f\in L^q(B_{2R_0}\cap \partial \Omega, d\nu)$ for some $q$ with 
\[
{q_0:=}\tfrac{Q_\mu-\Theta}{p-\Theta}<q{\le \infty},
\]
then any 
solution of the Neumann problem $u$ is {$\eps$}-H\"{o}lder continuous in $B_{R_0}$ with
\[
\eps=\min\left\{\beta_0, {\left(1-\frac{\Theta}{p}\right)\left(1-\frac{q_0}{q}\right)}\right\},
\]
where $\beta_0>0$ is the H\"older exponent for the interior regularity estimates 
established in~\cite[Theorem~5.2]{KiSh}. 
 More specifically, one has the estimate
\begin{equation}
\sup_{x,y\in B_{R_0},\ x\ne y}\frac{|u(x)-u(y)|}{d(x,y)^\eps}\le C,
\end{equation}
where $C>0$ is a constant depending on the structure constants, on the choice of $q$, 
on $\|f\|_{L^q(B_{2R_0}\cap \partial \Om)}$ and on $R_0$.

\item { (Harnack inequality)} There is a constant
$C>1$ depending only on the structure constants, such that if $u\ge 0$ on ${\partial\Om}$ and
$W\subset\partial\Om$ is a non-empty relatively open subset of $\partial\Om$ with $f=0$ on $W$, then
whenever $x\in\Om\cup W$ and $r>0$ such that $B(x,2r)\cap\overline{\Om}\subset W\cup\Om$, we have
the Harnack inequality
\[
 \sup_{B(x,r)\cap\overline{\Om}}\, u\le C\, \inf_{B(x,r)\cap\overline{\Om}}\, u.
\]
\end{enumerate}
\end{theorem}

\begin{remark}\label{regularity} The proof of existence and stability
of solutions  to the Neumann problem represents the bulk of this paper, and it is the main  contribution of our present work.
Note that the Harnack inequality and the H\"older continuity up to the boundary follow from local arguments, 
and as such their proof is the same as the proof of the analogous 
properties for bounded domains in~\cite{CKKSS}.  
We also remark that in the setting of the Euclidean metric, with Lebesgue measure,  the range for the exponent $q$ 
coincides with the range found by Caffarelli and Stinga in~\cite{CSt} where they studied the case $p=2$. See 
also \cite{CabSire1, CabSire2} for related results. For the optimality of this higher integrability exponent see subsection~\ref{sub:Holder}.
\end{remark}

\begin{remark}\label{regularity1}
 In regions where $f$ is non-negative, 
the results of~\cite{Mak} yield a better H\"older exponent for solutions than the one in Theorem \ref{intro-thm:Neumann-regularity},
namely,
\[
\epsilon=\min\bigg\lbrace \beta_0,\, \left(1-\frac{q_0}{q}\right)\, \left(\frac{p-\Theta}{p-1}\right)\bigg\rbrace.
\]
 This exponent is also optimal, given the hypothesis on $f$.
  The argument in \cite{Mak}  relies on Caccioppoli-type estimates
which are not available in regions where $f$ changes sign. In the Euclidean setting, the results in~\cite{ono} lead to 
the same (optimal) H\"older exponent without any restrictions on the sign of $f$. Recently we were able to extend the results of \cite{ono} to our more general setting, yielding the same optimal exponent. This work will appear in a separate paper. 
 See subsection~\ref{sub:Holder} for more details on this.
\end{remark}

\begin{remark}  We note that the problem of the optimal regularity of weak solutions for zero right-hand side, which is arguably one of  the most interesting regularity question related to $p-$Laplacian-like equations, is not addressed in this paper. The reason is that, unless additional curvature-type hypotheses are imposed, Lipschitz regularity fails even in the case $p=2$ (see \cite{KRS}).
\end{remark}

\bigskip

Next we turn to the study of fractional powers of the $p$-Laplacian, $(-\Delta_p)^\theta$, for $1<p<\infty$ on a 
doubling metric measure space $(Z, d, \nu)$. 
In the Euclidean setting\footnote{ Two excellent introductory papers for the Euclidean setting are \cite{Caff} and \cite{hitchicker}. See also \cite{FRRO} for a more exhaustive list of references and an historical perspective.
}  
this problem has been framed in a variational setting, and studied by many authors 
(see for instance \cite{AMRT, BBK, BLS, CKP, dTGCV, Gar, IMS,KKL, IN} and the references therein).
In these papers, weak solutions of the nonlocal homogeneous PDE $(-\Delta_p)^\theta u=0$ arise as minimizers of the homogeneous Besov norm 
$\Vert u\Vert_{\theta,p}^p:=\calE_{p,\theta}(u,u)$,  where
 \[
 \calE_{p,\theta}(u,v)=\int_{\R^n}\int_{\R^n}\frac{|u(y)-u(x)|^{p-2}(u(y)-u(x))(v(y)-v(x))}{\nu(B(x,d(x,y))\, d(x,y)^{\theta p}}\, dy\, dx.
 \]
Note that $\calE_{2,\theta}$ is a Dirichlet form in the sense of~\cite{FOT}. Several authors have 
extended this approach to non-Euclidean settings, see for instance \cite{BGMN, CG, FF, Gar, GT,PaPi} and 
the references therein.  However, this approach does not allow to use the Caffarelli-Silvestre method of studying 
the nonlocal operator via the Dirichlet-to-Neumann map for a local operator defined on a larger space. A remarkable exception is the paper \cite{SiVa}, where the authors use an approach similar to ours to define and study fractional powers of a large class of operators, including the Euclidean $p$-Laplacian in $\R^n$ and study rigidity properties arising from overdetermined problems. See also the  recent work \cite{GaVa} for more connections with rigidity properties. 
In \cite{CKKSS} and in the present paper, we study a related class of fractional $p$-Laplacians, 
corresponding to a Besov energy comparable to $\calE_{p,\theta}$, and for which the Caffarelli-Silvestre 
approach can be used. We will say more later in Section \ref{S:examples} about the relation between these 
two notions of fractional $p$-Laplacian. For now we just mention that in the case $p=2$, and in the Euclidean 
setting, our construction gives rise to the same operators  studied in \cite{CS} and in the references above.

We start by invoking a powerful uniformization result 
\cite{BBS} (see also Butler's work \cite{Bu1,Bu2} for  
non-compact versions of this result): 
given parameters $1<p<\infty$ and $0<\theta<1$,
every  doubling metric measure space $(Z,d_Z,\nu)$ arises as the boundary of a uniform domain $\Om$ that is equipped
with a measure $\mu$ so that the metric measure space $X=\overline{\Om}=\Om\cup Z$, together with
$Z=\partial\Om$, satisfies conditions~(H0),~(H1) and~(H2), with  $\Theta=p(1-\theta)$. The metric on $\partial \Omega$ 
is induced by the metric on $\Omega$ and while it may not coincide with the original metric $d_Z$ on $Z$, it is in the same bi-Lipschitz class. 

After choosing a Cheeger differential structure $\nabla$ on $\Om$, 
we show that 
for each  function $u$ in the homogeneous Besov space $HB^\theta_{p,p}(Z)$, one can find $\widehat{u}$, the unique 
Cheeger $p$-harmonic function in $D^{1,p}(\Om)$ such that
$\widehat{u}$ has  
trace $Tr(\widehat{u})=u$ $\nu$-almost everywhere on $Z$.

Following the strategy in \cite{CKKSS} we next
show that $\int_\Om|\nabla \widehat{u}|^p\, d\mu\approx \Vert u\Vert_{\theta,p}^p$ (see
Section~\ref{sec:construct-fractLap} below). We then set
\begin{equation}\label{eq:ET-go-home}
\mathcal{E}_T(u,v):=\int_\Om|\nabla\widehat{u}|^{p-2}\nabla\widehat{u}\cdot\nabla \widehat{v}\, d\mu.
\end{equation}

A function $u\in HB^\theta_{p,p}(Z)$ is in the domain of the fractional $p$-Laplacian operator
$(-\Delta_p)^{\theta}$ if there is a function $f\in L^{p'}(Z, \nu_J)$ such that the integral identity
\[
\mathcal{E}_T(u,\pip)=\int_Z\pip\, f\, d\nu
\]
holds for every  $\pip\in HB^\theta_{p,p}(Z)$. 
We then denote
\[
(-\Delta_p)^{\theta} u=f\in L^{p'}(Z,  \nu_J).
\]

\begin{remark} Note that if $f=0$, then this PDE is the Euler-Lagrange equation satisfied by critical points of the energy 
functional $\mathcal{E}_T(u,u)$. Since this energy is proved to be equivalent to the homogeneous Besov norm, then the 
minimizers of one energy are global quasi-minimizers of the other.
\end{remark}

\begin{remark} 
Our construction is not intrinsic to the metric and the measure structure of the doubling space $(Z,d, \nu)$. 
{ Instead, the construction is intrinsic to the metric measure structure of the domain $\Om$
that $Z$ is the boundary of. As shown by the hyperbolic filling constructions, it is possible for $Z$ to be the boundary
of more than one such domain $\Om$. Once the nonlocal trace operator has been constructed on $Z$, it becomes
possible to explore  how the geometries of  different domains $\Om$ affect the corresponding trace operators on $Z$.}
 In the case
 $\Om$ is the uniformization of a hyperbolic filling as considered in~\cite{BBS},
 there is a natural class of Cheeger differential structures that all give rise to the same fractional $p$-Laplacian operator on $Z$.
Such differential structures make use of the fact that $\Om$ is made up mostly of line-segments equipped with weights, and 
the choice of orientation on each line segment does not affect the outcome.
\end{remark}

\begin{remark}
 In the special case in which $p=2$ and $(Z,d,\nu)$ admits a Poincar\'e inequality, the set $\Omega$ can be constructed 
 through the product $Z\times \R$. As proved in \cite{CKKSS}, our construction gives rise to the same fractional operators 
 as in \cite{EbGKSS}. Since the latter generalizes the result of Caffarelli-Silvestre, then so does our construction.
\end{remark}

We summarize our results for the fractional $p$-Laplacian in the following theorem. In the statement, $Q_\nu$ 
denotes the lower mass bound exponent~\eqref{eq:lower-mass-exp} for the 
doubling measure $\nu$ on $Z$, and  $p'=p/(p-1)$ is the H\"older conjugate of $p$.

\begin{theorem}\label{thm:main-fract-Lap-intro}
Suppose that $(Z,d_Z,\nu)$ is a complete, uniformly perfect, unbounded doubling metric measure space,
and $0<\theta<1$. Then
the
form $\mathcal{E}_T$ 
given by~\eqref{eq:ET-go-home} satisfies
$\mathcal{E}_T(u,u)\approx\mathcal{E}_{p,\theta}(u,u)$ for each $u\in HB^\theta_{p,p}(Z)$
with the comparison constant depending solely on the doubling constant of $\nu$ and the indices $p,\theta$. If we denote by
$(-\Delta_p)^\theta$ the fractional $p$-Laplacian associated to the form $\mathcal{E}_T(u,u)$, then the following hold. 
\begin{enumerate}
\item[{\rm (i)}] For each $f\in L^{p'}(Z,\nu)\cap L^{p'}(Z, \nu_J)$ with $\int_Zf\, d\nu=0$ there exists a function $u_f\in HB^\theta_{p,p}(Z)$ such that
$$(-\Delta_p)^\theta u_f=f$$ on $Z$. If $\widetilde{u_f}$ is any other such function,
then $u_f-\widetilde{u_f}$ is constant $\nu$-a.e.~in $Z$. If, in addition, $f\in L^q(Z)$ for some
$q>\max\{1, Q_\nu/\theta\}$, then $u_f$ is H\"older continuous on $Z$
{with the H\"older exponent depending solely on $p$, $\theta$, $q$, and the doubling constant of $\nu$}. 
\item[{\rm (ii)}] There exists a constant $C>0$, depending only on the structural constants, such that if 
$f_1,f_2\in L^{p'}(Z, \nu)\cap L^{p'}(Z, \nu_J)$ with
$\int_{Z}f_1\, d\nu=0=\int_{Z}f_2\, d\nu$
and $u_{f_1}$, $u_{f_2}$ are the functions in $HB^\theta_{p,p}(Z)$ corresponding to $f_1$, $f_2$ as above, then 
\begin{multline}\label{eq:control-stable}
\Vert  u_{f_1}- u_{f_2}\Vert_{HB^\theta_{p,p}(Z)}\\ \le C\, \max\{\Vert f_1\Vert_{L^{p'}(Z,\nu_J)},\Vert f_2\Vert_{L^{p'}(Z,\nu_J)}\}^\kappa\ 
\times \\ \times 
\Vert f_1-f_2\Vert_{L^{p'}(Z,\nu_J)}^\tau,
\end{multline}
with $\kappa=0$, $\tau=1/(p-1)$ when $p\ge 2$ and $\kappa=(2-p)/(p-1)$, $\tau=1$ when $1<p<2$.
\item[{\rm (iii)}] Let $W\subset Z$ be an open (nonempty) subset such that $f=0$ on $W$.
There exists a constant $C>0$, depending only on the 
structural constants, such that if $u$ is a  solution of $(-\Delta_p)^\theta u=f$ in $Z$ with $u\ge 0$ on $Z$, 
then $u$ satisfies the Harnack inequality
$ \sup_B u \le C \inf_B u$ for all balls $B=B_R$ such that  $B_{4R}\subset W$. 
Note that by assuming that $\int_Zf\, d\nu=0$ we implicitly also assume that $f\in L^{1}(Z,\nu)$.
\end{enumerate}
\end{theorem}

 This is a generalization of the version of this theorem for bounded $Z$ which is  proved in \cite{CKKSS}.  See 
 Section~\ref{S:examples} for a study of the relation between these results, and those in the Euclidean literature.

\begin{remark} 
It is an interesting open problem whether the solutions still exist if the integrability of the data is relaxed to 
$f\in L^{p'}(\partial\Om, \nu)$. 
\end{remark}

\begin{remark}
In the setting of Euclidean spaces, some aspects of Theorem~\ref{thm:main-fract-Lap-intro} can be found in currently extant literature
for Besov energy minimizers; these are not known to arise as trace operators except in the case that $p=2$. The papers discussed
in this remark are only a small but representative sampling of available literature on the subject; we cannot hope to give an exhaustive
list of papers on the topic here.

Existence of Besov energy minimizers in bounded domains in $\R^n$ was established in~\cite{IN} when the inhomogeneity data
$f$ is continuous on the closure of the domain and the Dirichlet boundary data is continuous on the entire boundary of the domain.
Local regularity properties of Besov energy
minimizers were established in~\cite{IMS} only for bounded domains in $\R^n$ with smooth boundary and bounded inhomogeneity data 
$f$ on the boundary of the domain. The existence result was extended to more general bounded Lipschitz domains in $\R^n$
with more general inhomogeneity data and Dirichlet boundary data
in~\cite{GW}.
In considering the Besov energy minimization problem in the domain, \cite{IMS} imposes zero Dirichlet data on the boundary (or more
specifically, in the complement) of the domain. A weak version of~(ii) of Theorem~\ref{thm:main-fract-Lap-intro} was established in the 
setting of Euclidean domains in~\cite[Section~3]{IN}, where quantitative controls as in~\eqref{eq:control-stable} were not established;
however, the weak stability property established in~\cite{IN} sufficed to develop a Perron-type solution method, and~\cite{IN} then
continued on to establish existence of solutions for more general inhomogeneity data and boundary data provided they are 
resolutive.

Thus, to the best of our knowledge,
existence and uniqueness of inhomogeneous Besov energy minimizers on entire $\R^n$ seems to have not been established
in currently extant literature, nor the quantitative stability~\eqref{eq:control-stable} seems to be proven. For domains in
a compact doubling metric measure space, the Dirichlet boundary value problem associated with the non-local trace operator
was explored in~\cite{KLS}, where an extension of this stability was established as well.
\end{remark}

\begin{remark}
We have so far compared our results to the extant literature in PDEs. We point out here that the theory of nonlocal energies predates
the work of Caffarelli and Silvestre~\cite{CS}, and has a longer established history in the field of probability. We are grateful to
Mathav Murugan for pointing out the work of Molchanov and Ostrovskii~\cite{MO}. In the context of probability and Dirichlet forms~\cite{FOT},
the ($p=2$) nonlocal energies are associated with jump processes, and have been extensively studied, see for
a sampling the papers by Chen, Kumagai, Kim, and Wang~\cite{CK1, CK2, CKK, CKW}. The paper~\cite{CKW2} considers a large
class of jump processes for which the properties of local H\"older continuity and Harnack inequality properties of solutions compare
with the properties of associated heat kernels.
\end{remark}

\medskip
	
{\bf Acknowledgments.} The research of N.S. is partially supported by NSF grants~\#DMS-2054960 and~\#DMS-2348748.
 The research of L.C. is partly funded by NSF~\#DMS-195599 and \#DMS-2348806. The research of R.K. is supported by the Research Council of Finland grant 360184. 
Part of the research in this paper was conducted when the authors were at the BIRS 
workship \emph{Smooth Functions on Rough Spaces and Fractals with Connections to 
Curvature Functional Inequalities} in Banff, Alberta. The authors wish to thank that august institution for
its kind hospitality and for the availability of meeting rooms with blackboards, and to Mathav Murugan who
kindly informed us of the reference~\cite{MO} there. The first named author gave a series of lectures based upon the results
in this paper at the 13th School on Analysis and Geometry in Metric Spaces at the University of Trento, in July 2024. Some of the  
questions raised by the audience have informed the discussion in
Section~\ref{S:examples}. In particular, the  authors are  grateful to Andrea Pinamonti for helpful conversations. Finally, we wish 
to thank Yannick Sire and Enrico Valdinoci for sharing insightful feedback and references.

\section{Strategy of the proof}\label{Sec:strategy}

The extension of the results in \cite{CKKSS} to the non compact setting is non-trivial and  as the proofs are rather technical, 
we felt it might be useful to summarize in this section all the main steps, and introduce the notation used.

\bigskip

Here $\mu$ is a doubling measure on $\Om$, and it has lower mass bound exponent $Q_{\mu}$ as defined in \eqref{eq:lower-mass-exp}.

Let $\Om$ be an unbounded uniform domain with $\partial\Om$ also unbounded. 
We fix a point $x_0\in\partial\Om$
and set  
$B_0=\{x\in\overline{\Om}\, :\, d(x,x_0)<1\}$. For each positive integer $k$ we set
$2^kB_0:=\{x\in\overline{\Om}\, :\, d(x,x_0)<2^k\}$, and set 
$\overline{2^{k}B_0}=\{x\in\overline{\Om}\, :\, d(x,x_0)\le 2^k\}$.  

We adopt the procedure outlined in Subsection~\ref{previous work} below.
For positive integers $k$, we set 
\[
\Om_k=\Om\cup(\partial\Om\setminus \overline{2^kB_0})=\Om\cup\{x\in\partial\Om\, :\, d(x,x_0)>2k\}.
\]
Note that then $\Om_k$ is open in $\overline{\Om}$, and
\[
\partial\Om_k=\partial\Om\cap\overline{2^kB_0}.
\]
We fix $\beta>0$ such that $\beta p>Q_{\mu}$, and set $\pip:[0,\infty)\to[0,\infty)$ by
\[
\pip(t)=\min\{1, t^{-\beta}\}.
\]
For an open connected set $U$ that is locally compact but not complete, we set $\partial U=\overline{U}\setminus U$
where $\overline{U}$ is the metric completion of $U$. Let   $d_U(x):=\inf\{d(x,y)\, :\, y\in \partial U\}$.
Consider the metric $d_{\pip}$ on $U$ given by
\[
d_{\pip}(z,w)=\inf_\gamma\int_\gamma \pip\circ d_U\, ds
=\inf_\gamma\int_0^{\ell(\gamma)}  \pip(d_U(\gamma(t)))\, dt,
\]
where we have, without loss of generality, assumed the curves $\gamma$ to be arc-length parametrized, we have 
denoted by $\ell(\gamma)$ the length of $\gamma$, 
 and the infimum
is over all rectifiable paths $\gamma$ in $U$ with end points $z$ and $w$. Let
$\overline{A}^\pip$ denote the completion of $A\subset\overline{U}$ with respect to the metric $d_{\pip}$.

With this, we set 
\[
\Om_{k,\pip}:=\overline{\Om_k\cup\partial\Om_k}^\pip\setminus\partial\Om_k=
\Om_k\cup\{\infty\}.
\]
Note that $\partial\Om_k=\overline{2^kB_0}\cap\partial\Om$ is a bounded set. The analogous 
metric is denoted by $d_{k,\pip}$, and moreover, the transformed 
measure is denoted $\mu_{k,\pip}$, and is given by
\[
\mu_{k,\pip}(A)=\int_A\pip(d_{\Om_k}(x))^p\, d\mu.
\]

 Our strategy is structured through the following  list of  steps leading to the solution of the Neumann problem:
\begin{enumerate}
\item Solve the existence/uniqueness problem for the Dirichlet boundary value problem on $\Om$, with boundary data
$f\in HB^{1-\Theta/p}_{p,p}(\partial\Om)$ and find a solution in $D^{1,p}(\Om)$. 
\item Let $f:\partial\Om\to\R$ be a function in $L^{p'}(\partial\Om,\nu)$ with support in $B_0\cap\partial\Om$
with $\int_{\partial\Om}f\, d\nu=0$. We 
solve the Neumann boundary value problem on  $\Om_{2,\pip}$ 
for this boundary data.
Show that then this solution also is a solution for the Neumann boundary value problem on $\Om$.
We also need to show uniqueness for the case that $f$ is compactly supported.

\item For more general $f\in L^{p'}(\partial\Om, \nu_J)$, 
not necessarily compactly supported, we will solve
the Neumann boundary value problem on $\Om_k$ for each $k\in\N$ with 
appropriately defined approximations $f_k$ of the Neumann boundary data $f_k\to f$ supported in $2^kB_0$. 
\item Next look at the sequence of such solutions $u_k$, one for each $k\in\N$, and  to exploit 
compactness to show that there is a limit function that solves
the Neumann boundary value problem with data $f$ itself. 
\item Prove uniqueness modulo a constant. This is achieved by showing continuity of the solution in 
the $D^{1,p}$ seminorm with respect to the data in the weighted $L^{p'}$ spaces.
\end{enumerate}

\section{Background definitions and results}\label{Sec:background}

\subsection{Doubling property and codimensionality}\label{subsect:doubling}
A measure $\mu$ is \emph{doubling} if it is a Radon measure and there is a constant $C_d\ge 1$ such that
\[
0<\mu(B(x,2r))\le C_d\, \mu(B(x,r))<\infty
\]
for each $x\in X$ and $r>0$.
Doubling measures satisfy the following lower mass bound property: there are constants $c>0$
and $Q_\mu>0$ depending only on $C_d$,  such that
for each $x\in X$, $0<r<R$, and for each $y\in B(x,R)$,
\begin{equation}\label{eq:lower-mass-exp}
c\left(\frac{r}{R}\right)^{Q_\mu}\le \frac{\mu(B(y,r))}{\mu(B(x,R))},
\end{equation}
see for example~\cite[page~76]{HKST}. 

 Given a Radon measure $\mu$ on a domain $\Om$ and a Radon measure $\nu$ on $\partial\Om$,
we say that $\nu$ is \emph{$\Theta$-codimensional with respect to $\mu$} if there is a constant $C\ge 1$ such that 
whenever $0<r<2\diam(\partial\Om)$ and $\xi\in\partial\Om$, we have 
\begin{equation}\label{eq:codim}
\frac{1}{C}\, \frac{\mu(B(\xi,r))}{r^\Theta}\le \nu(B(\xi,r))\le C\, \frac{\mu(B(\xi,r))}{r^\Theta}.
\end{equation}

\subsection{Sobolev-type spaces and Poincar\'e inequalities}\label{subsect:sob-PI}
One of the main features of a first order calculus in metric measure spaces was first developed by
Heinonen and Koskela~\cite{HK}:  given a measurable function $u:X\to\R$, we say that a non-negative Borel measurable function $g$ on
$X$ is an \emph{upper gradient} of $u$ if
\[
 |u(x)-u(y)|\le \int_\gamma g\, ds
\]
for every non-constant compact rectifiable curve $\gamma$ in $X$; here, $x$ and $y$ denote the terminal points of $\gamma$.

The function $u$ is said to be in the homogeneous Sobolev space $D^{1,p}(X)$ 
if $u$ has an upper gradient that belongs to $L^p(X)$;
and $u$ is said to be in the Newton-Sobolev class $N^{1,p}(X)$ if it is in $D^{1,p}(X)$ and in addition, 
$\int_X|u|^p\, d\mu$ is finite. 
Given that upper gradients are not unique, we set the energy seminorm on $D^{1,p}(X)$ by
\begin{equation}\label{penergy}
\mathcal{E}_p(u)^p:=\inf_g\int_Xg^p\, d\mu,
\end{equation}
where the infimum is over all upper gradients $g$ of $u$. The norm on $N^{1,p}(X)$ is given by
\[
\Vert u\Vert_{N^{1,p}(X)}:=\Vert u\Vert_{L^p(X)}+\mathcal{E}_p(u).
\]

If $1\le p<\infty$, for each $u\in D^{1,p}(X)$ there is a unique (up to sets of $\mu$-measure zero) non-negative function $g_u$
that is the $L^p$-limit of a sequence of upper gradients of $u$ from $L^p(X)$ and so that for each upper gradient $g$ of $u$ we
have that $\Vert g_u\Vert_{L^p(X)}\le \Vert g\Vert_{L^p(X)}$. The functions $g_u$ belong to a larger class of ``gradients lengths'' of $u$,
called $p$-weak upper gradients, see for example~\cite{BBbook, HajK, HKST, Sh}. This 
function $g_u$ is said to be the \emph{minimal $p$-weak upper gradient} of $u$.

For $1\le p<\infty$, the metric measure space $(X,d,\mu)$ is said to support a $p$-Poincar\'e inequality if there are constants
$C_P>0$ and $\lambda\ge 1$ such that for all $u\in D^{1,p}(X)$ and balls $B=B(x,r)$ in $X$, we have
\[
\jint_{B(x,r)}|u-u_B|\, d\mu\le C_P\, r\, \left(\jint_{B(x,\lambda r)}g_u^p\, d\mu\right)^{1/p}.
\]
As shown in~\cite{HajK}, if
$X$ is a length space, then we can take $\lambda=1$ at the expense of increasing the constant
$C_P$. It is also well known that the $p$-Poincar\'e inequality and the doubling property of the measure
imply the $(p,p)$-Poincar\'e inequality
 as follows:
\[
\jint_{B(x,r)}|u-u_B|^p\, d\mu\le C_{SP}\, r^p\, \jint_{B(x,r)}g_u^p\, d\mu,
\]
see for instance~\cite{HajK, HKST}. 

\subsection{Besov spaces}\label{subsec2.4}

Consider a metric measure space $(Z,d,\nu)$.  For $0<\theta<1$ and $1<p<\infty$ we will consider the following
Besov energy:
\[
\Vert u\Vert_{\theta,p}^p:=
\int_{Z}\int_{Z}\frac{|u(y)-u(x)|^p}{d(x,y)^{\theta p}\, \nu(B(x,d(x,y)))}\, d\nu(y)\, d\nu(x),
\]
and set $B^\theta_{p,p}(Z)$ to be the space of all $L^p$--functions for which this energy is finite.
If $\nu$ is Ahlfors $Q$-regular, then
we can replace $\nu(B(x,d(x,y))$ with $d(x,y)^Q$ to obtain an equivalent energy. The following theorem
holds also for this modified norm. 

The homogeneous Besov space $HB^\theta_{p,p}(\partial\Omega)$ is the collection of equivalence classes of functions from
$B^\theta_{p,p}(\partial\Omega)$, where two functions $u,v\in B^\theta_{p,p}(\partial\Omega)$ are said to be equivalent
if $\Vert u-v\Vert_{\theta,p}=0$, that is, $u-v$ is $\nu$-a.e.~constant on $\partial\Omega$. By a slight
abuse of notation, we conflate equivalence classes that are elements of $HB^\theta_{p,p}(\partial\Omega)$
with representative functions in those classes, but are careful to remember that then there is an ambiguity up to additive
constants here.

\subsection{Trace and extension theorems}\label{s:trace}
In this section we describe  bounded linear trace  operators 
$T:D^{1,p}(\Omega,\mu)\rightarrow B^{1-\Theta/p}_{p,p}(\partial\Omega,\nu)$.  In the case in which $\partial \Omega$ is bounded, their existence  
 is proved in \cite[Proposition 8.3]{GKS} and follows from the work of Mal\'y \cite{Mal} for bounded 
 uniform domains by passing through the transformation $(\Omega,d,\mu)\mapsto(\Omega_\pip,d_\pip,\mu_\pip)$.

 In this paper, however, we are interested in unbounded domains with unbounded 
 boundaries, and so such trace result is insufficient. However, if we weaken the 
 trace class to the homogeneous $HB^{1-\Theta/p}_{p,p}(\partial\Omega,\nu)$, 
 which is a more natural class of boundary data when $\partial\Omega$ is 
 unbounded, we have the following result from \cite{GS2}.

\begin{theorem}\label{thm-trace-local-est} 
There exists a bounded linear trace operator $T:D^{1,p}(\Om,\mu)\to HB^{1-\Theta/p}_{p,p}(\partial\Om,\nu)$ such that
\[
\| Tu \|_{HB^{1-\Theta/p}_{p,p}(\partial\Om,\nu)}\lesssim \| u \|_{D^{1,p}(\Omega,\mu)}
\] 
for all $u\in D^{1,p}(\Omega,\mu)$. Here   $\Theta$  is the exponent in \eqref{eq:codim} and  $Tu:\partial\Om\to\R$ is given by
\[
\lim_{r\to 0^+}\frac{1}{\mu(B(\xi,r)\cap\Om)}\, \int_{\mu(B(\xi,r)\cap\Om)}|u-Tu(\xi)|\, d\mu=0
\]
for $\nu$-a.e.~$\xi\in\partial\Om$. Moreover, there is a bounded linear extension operator 
$E:HB^{1-\Theta/p}_{p,p}(\partial\Om,\nu)\to D^{1,p}(\Om,\mu)$ such that $T\circ E$ is the identity operator on 
$HB^{1-\Theta/p}_{p,p}(\partial\Om,\nu)$.
\end{theorem}

\subsection{Uniform domains}\label{subsect:Uniform}
Let $(X,d)$ be a locally compact and non complete metric space $(\Omega, d)$. Denote by $\partial \Omega=\overline \Omega\setminus \Omega$.
Following \cite{MS} (see also \cite{Jones, BHK}), $(X,d)$ is a {\it uniform space} if 
there exists a constant $C_U>0$ such that for every pair 
$x,y\in \Om$ one can find a 
curve $\gamma$ from $x$ to $y$ such that 
${\ell(\gamma)} \le C_U d(x,y)$ and such that for each $z$ in the 
image of $\gamma$ one has 
\[
\min \{\ell(\gamma_{x,z}), \ell (\gamma_{z,y})\} \le C_U d(z, \partial \Omega),
\]
where we have denoted by $\gamma_{a,b}$ 
a subcurve of $\gamma$ with end points $a$ and $b$, and  $\ell(\gamma)$ is the length of $\gamma$.
Here, if $\gamma$ has more than one subcurve with end points
$x,z$ or with end points $z,y$, then the above inequality is presumed to hold for each such subcurve.

\subsection{Differentiable structures}\label{Sec:Cheegerishness}
Some of the properties we are interested in depend on the existence of an Euler-Lagrange equation satisfied by energy minimizers. 
To achieve that we will be using Cheeger differentiable structures (see \cite[Theorem~4.38]{Che}).
A metric measure space $(\Om,d,\mu)$ is said to support a Cheeger differential structure of dimension $N\in\N$ if there exists a 
collection of coordinate patches $\{(\Om_\alpha,\psi_\alpha)\}$ and a $\mu$-measurable inner product 
$\langle \cdot,\cdot\rangle_{x}$, $x\in \Om_\alpha$, on $\R^N$ and a constant $C\ge 1$ such that
\begin{enumerate}
\item each $\Om_\alpha$ is a measurable subset of $\Om$ with positive measure and 
$\bigcup_\alpha \Om_\alpha$ has full measure;
\item each $\psi_\alpha:\Om_\alpha\rightarrow\R^N$ is Lipschitz;
\item for every function $u\in D^{1,p}(\Om)$, for $\mu$-a.e. $x\in \Om_\alpha$ there is a vector $\nabla u(x)\in\R^N$ such that
\[
\elimsup\limits_{\Om_\alpha\ni y\rightarrow x}\frac{|u(y)-u(x)-\nabla u(x)\cdot\left[\psi_\alpha(y)-\psi_\alpha(x)\right]|}{d(y,x)}=0,
\]
where 
$\nabla u(x)\cdot\left[\psi_\alpha(y)-\psi_\alpha(x)\right]=\sum_{i=1}^N(\nabla u(x))_i\, [\psi_{\alpha, i}(x)-\psi_{\alpha,i}(y)]$ 
denotes the Euclidean inner product, see for instance~\cite[Corollary~13.5.5]{HKST},

\item for every function $u\in D^{1,p}(\Om)$, for $\mu$-a.e. $x\in \Om_\alpha$ we have that
\[
\frac{1}{C}\, g_u(x)^2\le \langle \nabla u(x), \nabla u(x)\rangle_x\le C\, g_u(x)^2
\]
where $g_u$ is the minimal $p$-weak upper gradient of $u$.
\end{enumerate} 

When the metric $d$ is doubling, we may assume that the collection of coordinate patches is countable and 
that the coordinate 
neighborhoods $\{\Om_\alpha\}$ are pairwise disjoint. Note that there may be more than one possible 
Cheeger differential structure on a given space. 

A function $u\in D^{1,p}(\Om)$ is a \emph{Cheeger $p$-harmonic
function} in $\Om$ if, whenever $v\in D^{1,p}(\Om)$ has compact support in $\Om$, we have
\[
\int_{\text{supt}(v)}|\nabla u|^p\, d\mu\le \int_{\text{supt}(v)}|\nabla (u+v)|^p\, d\mu.
\]
Equivalently, we have the following corresponding Euler-Lagrange equation:
\[
\int_\Om|\nabla u(x)|^{p-2}\langle \nabla u(x),\nabla v(x)\rangle_x\, d\mu(x)=0.
\]
For brevity, in our exposition we will suppress the dependence of $x$ on the inner product structure, and denote
\[
\langle \nabla u(x),\nabla v(x)\rangle_x=:\nabla u(x)\cdot\nabla v(x)
\]
when this will not lead to confusion.
Cheeger $p$-harmonic functions are quasiminimizers of the 
$p$-energy~\eqref{penergy} in the sense of Giaquinta, and hence
we can avail ourselves of the properties derived in~\cite{KiSh}.

\subsection{Conformal invariance of $p$-harmonicity}\label{sec:conformal invariance} 

If $\Om$ is as in the setting of the paper, then it supports a Cheeger differential structure with the 
property that $|\nabla u(x)|\approx g_{u,d}(x)$ for $\mu$-a.e. $x\in\Om$, where 
$|\nabla u(x)|^2:=\langle\nabla u(x),\nabla u(x)\rangle_x$
and $g_{u,d}$ is the minimal $p$-weak upper gradient of the function $u$, with respect to the metric $d$. 

Next we explain how the transformation described in Section~\ref{Sec:strategy} above
affects a possible choice of the Cheeger differentiable structure.
Note that in Condition~(3) above, we could just as well have replaced $d(x,y)$ with $d_\pip(x,y)$,
as
\[
\lim_{y\to x}\frac{d(x,y)}{d_\pip(x,y)}=\frac{1}{\pip(d_{\Om_k}(x))},
\]
with $0<\pip(d_{\Om_k}(x))<\infty$ for each $x\in\Om_k$. Note also that as $d_\pip$ and $d$ are locally bi-Lipschitz equivalent on $\Om_k$,
it follows that $\psi_\alpha$ is also Lipschitz with respect to the metric $d_\pip$ by ensuring that each 
measurable set $\Om_\alpha$ is contained in a ball whose closure is contained in $\Om_k$. That is,
\begin{align*}
&\elimsup\limits_{\Om_\alpha\ni y\rightarrow x}\frac{|u(y)-u(x)-\nabla u(x)\cdot\left[\psi_\alpha(y)-\psi_\alpha(x)\right]|}{d_{k,\pip}(y,x)}\\
&=
\elimsup\limits_{\Om_\alpha\ni y\rightarrow x}\frac{|u(y)-u(x)-\nabla u(x)\cdot\left[\psi_\alpha(y)-\psi_\alpha(x)\right]|}{d(y,x)}\frac{d(y,x)}{d_{k,\pip}(y,x)}
=0.
\end{align*}
In conclusion, the transformed space $\Om_{k,\pip}$
also supports the Cheeger differential structure 
given by $\nabla_{k,\pip} u=\nabla u$, 
with the coordinate maps $\phi_{k,\alpha}=\psi_\alpha$ and 
the choice of inner product  given by
\[
\langle \cdot,\cdot\rangle_{x,k,\pip}:=
\frac{1}{\pip(d_{\Om_k}(x))^2}\langle \cdot,\cdot\rangle_{x}.
\]
It follows by~\cite[Equation~(7.2)]{GKS} that
\begin{equation}\label{D:nabla phi}
|\nabla_{k,\pip} u(x)|_\pip=\frac{1}{\pip\circ d_{\Om_k}(x)}|\nabla u(x)|\approx\frac{1}{\pip\circ d_{\Om_k}(x)}g_{u,d}(x)=g_{u,k,\pip}(x)
\end{equation}
for $\mu$-a.e. $x\in\Om$. Here, by $|\nabla_{k,\pip} u(x)|_{k,\pip}$ we mean 
$\sqrt{\langle \nabla_{k,\pip} u(x),  \nabla_{k,\pip} u(x)\rangle_{x,k,\pip}}$.
Given this transformation, we note that if $u,v\in D^{1,p}(\Om)$, then
\begin{align}\label{eq:ChEnergy-trans}
\int_{\Om_{k,\pip}} |\nabla_{k,\pip} u|_{k,\pip}^{p-2}\langle \nabla_{k,\pip} u(x), \nabla_{k,\pip} v(x)\rangle_{x,k,\pip} \, d\mu_{k,\pip}(x)
 &= \int_{\Om_k} |\nabla u|^{p-2} \nabla u\cdot\nabla v \, d\mu\notag\\
 =&\int_\Om |\nabla u|^{p-2} \nabla u\cdot\nabla v \, d\mu,
\end{align}
where, in the last equality, we have used the fact that $\mu(\Om_k\setminus\Om)=\mu((\partial\Om)\setminus B(x_0,k))=0$
because of the co-dimensionality relation between the measure $\nu$ on $\partial\Om$ and the measure $\mu$ on $\Om$
(indeed, any extension of $\mu$ as a Radon measure to $\overline{\Om}$ that preserves the doubling property will see
$\partial\Om$ as a null-measure set).
Thus, the $p$-harmonicity with respect to the Cheeger structure on $(\Om,d,\mu)$ 
is \emph{the same notion as} $p$-harmonicity
with respect to the transformed Cheeger structure on $(\Om_{k,\pip},d_{k,\pip},\mu_{k,\pip})$ for each positive integer $k$.

\subsection{An equivalent formulation of the Neumann problem}\label{sub:Def-Neumann}

Next, we discuss equivalent formulations for the Neumann problem and a version of a result from~\cite[Theorem~1.5]{CKKSS} 
that will be helpful in the following.

We begin by showing a continuity result that plays a crucial role in the proof of the existence of the solution.

\begin{lemma}\label{panacea-bis} Let $(Z,d,\nu)$ be a doubling metric measure space and fix $x_0\in \partial Z$. 
Assume that $f\in L^{p^\prime}(Z, \nu)$, with zero average $\int_Z f(x) d\nu(x)=0$ and that
$\int_Z |f(x)|^{p^\prime} J(x,x_0) d\nu(y) <\infty$, where
\begin{equation}\label{weight-v-bis}
J(x,y):=d(x,y)^{p'\theta }\nu(B(x,d(x,y)))^{p'/p}.
\end{equation}
There exists a constant $C>0$ depending only on $x_0$ and on the structure constants such that for every 
$v\in HB^\theta_{p,p}(Z)$, 
    
\begin{equation}\label{eq-panacea-bis}
\int_{Z} f(x) v(x) d\nu (x) \le  C\left(\int_{Z}|f|^{p^\prime}J(x,x_0)\, d\nu(x)\right)^{1/p^\prime}\, 
\|v\|_{HB^\theta_{p,p}(Z)}.
\end{equation}

\end{lemma}
\begin{proof}
Since $f$ has zero average, it follows that for every $y\in Z$ one has
$\int_{Z} f(x) v(x) d\nu (x)=\int_{Z}f(x)[v(x)-v(y)]\, d\nu(x)$, and so
 we obtain
\begin{align*}
\int_{Z} f(x) v(x) d\nu (x)&\le \int_{Z}J^{\frac{1}{p^\prime}}(x,y)|f(x)|\, \frac{|v(x)-v(y)|}{d(x,y)^\theta \nu(B(x,d(x,y)))^{1/p}}\, d\nu(x)\\
&\le \left(\int_{Z}|f|^{p^\prime}J(x,y)\, d\nu(x)\right)^{1/p^\prime}\, 
\bigg(\int_{Z}\frac{|v(x)-v(y)|^p}{d(x,y)^{\theta p}\nu(B(x,d(x,y)))}\, d\nu(x)\bigg)^{\frac{1}{p}}.
\end{align*}
Integrating both sides of the inequality in  $y\in B  $, for an arbitrary ball $B\subset Z$, and invoking H\"older inequality on the right hand side  yields

\begin{multline*}
\nu(B) \int_{Z} f(x) v(x) d\nu (x)  \\
\le \int_B  \Bigg[ \ \left(\int_{Z}|f|^{p^\prime}J(x,y)\, d\nu(x)\right)^{1/p^\prime} \, 
\left(\int_{Z}\frac{|v(x)-v(y)|^p}{d(x,y)^{\theta p}\nu(B(x,d(x,y)))}\, d\nu(x)\right)^{\frac{1}{p}}\Bigg] d\nu(y)
\\ 
\le \left(\int_B \int_{Z}|f|^{p^\prime}J(x,y)\, d\nu(x) d\nu(y) \right)^{\frac{1}{p'}}
\left(\int_B\int_{Z}\frac{|v(x)-v(y)|^p}{d(x,y)^{\theta p}\nu(B(x,d(x,y)))}\, d\nu(x) d\nu(y)
\right)^{\frac{1}{p}}
\\
\le 
\nu(B)^{\frac{1}{p'}} \left(\max_{y\in \overline B\cap Z} \int_{Z}|f|^{p^\prime}J(x,y)\, d\nu(x)\right)^{\frac{1}{p'}}
\|v\|_{HB^\theta_{p,p}(Z)}
\end{multline*}
Since the weight $J$ is a continuous function of $x$, the maximum is achieved at some point $x_0\in \overline B$. Clearly, 
one can substitute $x_0$ with another arbitrary point, at the cost of adding a multiplicative constant in front of the right 
hand side. To conclude the argument we note that we can choose $B$ so that $\nu(B)=1$, thus eliminating that dependence.
\end{proof}

This lemma allows us to make sense of the expression
$$\int_{\partial\Omega} f(x) \phi(x) d\nu(x)$$
when $f\in L^{p'}(\partial \Omega, \nu_J)$ and $\phi \in HB^{\theta}_{p,p} (\partial \Omega, \nu)$.
Recalling the trace theorems in Section \ref{s:trace} we can then state the following

\begin{theorem}\label{thm:equiv-unbounded}
Suppose that conditions~(H0),~(H1) and~(H2) hold for the metric measure space $(X,d,\mu)$ and $\Om$.
Fix $1<p<\infty$, and suppose that
$f\in L^{p'}(\partial\Om, \nu_J)\cap L^{p'}(\partial\Om, \nu)$ where $p'$ is the H\"older
conjugate of $p$ with $\int_{\partial\Omega}f\,d\nu =0$.
Let $u\in D^{1,p}(\Omega)=D^{1,p}(\overline{\Omega})$. Then
the following are equivalent.
\begin{enumerate}
\item[\rm{(a)}] $u$ is a solution to the Neumann boundary value problem with data
  $f$ in the domain $\Omega\subset X$; that is, for all $\phi\in D^{1,p}(\overline{\Omega})$,
      \[
      \int_{\Omega}|\nabla u|^{p-2}\nabla u\cdot\nabla\phi\,d\mu=\int_{\partial\Omega}\phi f\,d\nu
      \]
  \item[\rm{(b)}] $u$ minimizes the energy functional
        \[
        I(v):=\int_{\overline{\Omega}}|\nabla v|^{p}\,d\mu-p\int_{\partial{\Omega}}v f\,d\nu
        \]
        among all functions $v\in D^{1,p}(\overline{\Omega})$.  Here we extend $\nu$ to a measure
        on the closure $\overline{\Omega}$ by zero outside of $\partial\Omega$.        
 \end{enumerate}
\end{theorem}

This theorem was proved in \cite{CKKSS} in the case when $\Omega$ is bounded, however the 
proof of this equivalence extends immediately to the unbounded case. 

The following corollary (proved as  in \cite{CKKSS}) yields a connection to the Euclidean Neumann problem, and to Caffarelli-Silvestre's setting.

\begin{corollary}\label{normal-derivative} Assume the hypotheses of the previous theorem and let us continue to denote by $\nu$ the 
extension of $\nu$ to a measure
      on the closure $\overline{\Omega}$ which is zero outside of $\partial\Omega$.
  If $u$ is $p$-harmonic in $\Omega$, then one has
      \begin{equation}\label{limit}
      |\nabla u|^{p-2}\nabla u\cdot\nabla\eta_\epsilon\,d\mu\rightharpoonup -f\,d\nu.
      \end{equation}
      Here the convergence is that of weak convergence of signed Radon measures on $\overline{\Om}$,
      and the function $\eta_\epsilon:\overline{\Omega}\to\R$ is given by
$\eta_\eps(x)=\min\{1, d(x, \overline{\Om}\setminus\Omega)/\eps\}$. 
\end{corollary}

Note that in view of  our structural assumption~(H2), we have
$$\int_{B(\zeta,r)}|\nabla \eta_\eps|\, d\mu\le C\eps^{-1+\Theta}\nu(B(\zeta,2r))$$ whenever
$\zeta\in\partial\Om$ and $r>0$ for some constant $C>0$ depending only on the structural constants.

\begin{remark} 
In the bounded case, in \cite{CKKSS} it is shown that the limit identity \eqref{limit} is equivalent to $u$ being a solution of the 
Neumann problem. However in the unbounded setting this limit identity does not provide sufficient information about the 
behavior of $u$ at infinity to guarantee the missing implication.
\end{remark}

\subsection{Dirichlet problem  in unbounded domains with bounded boundary}\label{previous work}

In this section, we recall a few results from the recent articles \cite{GS} and \cite{GKS} 
concerning the Dirichlet boundary value problem in unbounded domains $ \Omega$, 
whose boundary $\partial \Omega$ is bounded.

Let $(\Om,d)$ be a locally compact, unbounded, metric space such that $\Om$ is 
a uniform domain in its completion $\overline{\Om}$ with bounded boundary 
$\partial\Om:=\overline{\Om}\setminus\Om$. We endow 
$\Om$ with a doubling measure $\mu$ such that the metric 
measure space $(\Om,d,\mu)$ supports a $p$-Poincar\'{e} 
inequality for some fixed $1\leq p<\infty$. 
 We also require the co-dimensionality hypothesis \eqref{eq:codim}.

We recall here again the transformation briefly described in Section~\ref{Sec:strategy} above.
Consider the {\it dampening function} $\pip:(0,\infty)\rightarrow(0,\infty)$ given by 
$\pip(t)=\min\{1,t^{-\beta}\}$ for $t>0$ and $\beta>1$ satisfying $\beta\,p>Q$,  
 where $Q$ is
the lower mass bound exponent of $\mu$
from~\eqref{eq:lower-mass-exp}.
Setting $d_\Omega(x)=d(x,\partial \Omega)$, we consider the new transformed metric 
$d_\pip$ on $\Om$ given by
\[
d_\pip(x,y):=\inf_\gamma \ell_\pip(\gamma):= \inf_\gamma\int_\gamma \pip(d_\Omega(\gamma(t)))\, dt,
\]
with the infimum ranging over all rectifiable curves $\gamma$, arclength-parametrized,
in $\Om$ with end points $x,y\in\Om$,  
and a new measure 
\[
d\mu_\pip(x)=\pip(d_\Omega(x))^p\, d\mu.
\]

Consider now $\Om_\pip:=\overline{\Om\cup\partial\Om}^\pip\setminus\partial\Om$, 
where $\overline{A}^\pip$ denotes the completion of $A\subset\overline{\Om}$ with respect to the metric $d_\pip$. 
It follows that $\partial{\Om_\pip}=\partial\Om$ and $\Om_\pip$ differs from $\Om$ by one point, which we denote 
by $\infty$, i.e. $\Om_\pip=\Om\cup\{\infty\}$. It is also true that $\overline{\Om_\pip}^\pip$ is compact because 
$\partial\Om$ is bounded. 
 From \cite{GS} and \cite{GKS}, we have the following:

\begin{theorem}\label{th:transformed}
The space $(\Om_\pip, d_\pip)$ is a bounded uniform domain and in addition,
$(\Om_\pip,d_\pip,\mu_\pip)$ is a doubling metric measure supporting a $p$-Poincar\'{e} inequality. 
\end{theorem}

Note that the results of~\cite{GKS} dealt with upper gradient-based $p$-harmonic functions, but
what was done in~\cite{GKS} hold \emph{mutatis mutandis} with the Cheeger derivative as well given the discussion in
Sections~\ref{Sec:Cheegerishness} and \ref{sec:conformal invariance} of the present paper.

From~\cite[equation~(7.1) and Proposition~7.4]{GKS} we have the following.

\begin{lemma}\label{lem:nopiptopip}
Under the hypotheses above,
 we have  
	$D^{1,p}(\Om)=N^{1,p}(\Om_\pip)=N^{1,p}(\overline{\Om_\pip}^\pip)$. 
\end{lemma}

Let $1<p<\infty$ and $\nu$ be a Borel regular measure on $\partial\Om$ that is $\Theta$-codimensional to $\mu$ for 
some $0<\Theta<p$. In the next theorem, we will let $Q_\pip^+$ and $Q_\pip^-$ denote the upper and lower mass bound 
exponents ``at infinity" of $\mu_\pip$. That is, $Q_\pip^+$ and $Q_\pip^-$ are the numbers 
satisfying 
\[
\left(\frac{r}{R}\right)^{Q^+_\pip} \gtrsim
\frac{ \mu_\pip(B_{d_\pip}(\infty,r)) }{ \mu_\pip(B_{d_\pip}(\infty,R)) }\gtrsim \left(\frac{r}{R}\right)^{Q^-_\pip}
\]
for all $0<r\leq R <r_0$, where $r_0$ is a threshold radius depending on the uniformity constant $C_U$ of $\Omega$ and 
the parameter $\beta$ from the definition of the function $\pip$.
In the  theorem below,  $T:D^{1,p}(\Omega,\mu)\rightarrow B^{1-\Theta/p}_{p,p}(\partial\Omega,\nu)$ will denote the bounded 
linear trace operator described in Section \ref{s:trace}.

\begin{theorem}\label{thm:bounded-Dirichlet} Denote by $\Theta$ the exponent in \eqref{eq:codim}.
If $\partial\Om$ is bounded, then for any $f\in B^{1-\Theta/p}_{p,p}(\partial\Om,\nu)$ there is a
function $u\in D^{1,p}(\Om,\mu)$ such that
\begin{enumerate}
\item $u$ is $p$-harmonic in $(\Om,d,\mu)$,
\item $Tu=f$ on $\partial\Om$ $\nu$-almost everywhere.
\end{enumerate}
If $p\leq Q_\pip^+$, then the solution $u$ is unique. If $p>Q_\pip^-$, then 
for each solution $u$ of the problem the limit $\lim_{\Om\ni y\to\infty}u(y)$ exists; this limit uniquely determines the solution.
Moreover, there is a unique $u$ satisfying the above two conditions such that for all $v\in D^{1,p}(\Om)$ with
$Tv=0$ $\nu$-a.e. on $\partial\Om$ we have
\[
\int_\Om |\nabla u|^{p-2} \nabla u\cdot \nabla v \, d\mu=0.
\]
\end{theorem}

\begin{remark} The main point of the last claim of the above theorem is that the function $v$ need not have 
compact support in $\Om$, that is, it need not vanish at $\infty$.
The distinct behavior corresponding to $p\le Q_\pip^+$ versus $p>Q_\pip^-$ is due to the $p$-parabolicity versus $p$-hyperbolicity
of $\Om$ at $\infty$.

If $u\in N^{1,p}(X)$, then necessarily $u\in L^p(X)$, and so effectively one can think of $u$ as vanishing at $\infty$ if $X$ is
unbounded with $\mu(X)=\infty$.
Functions in $D^{1,p}(X)$ have a more ambiguous behavior at $\infty$. Indeed, the results of~\cite{GKS} tell us that 
the behavior depends on whether $X$ is $p$-parabolic (in which case we can think of $\infty$ as having zero $p$-capacity)
or is $p$-hyperbolic, see for example~\cite[Proposition~7.6]{GKS}. The setting of~\cite{GKS} is that of an unbounded uniform
domain, but with a bounded boundary. In our setting, where the boundary itself is also unbounded, the discussion of~\cite{GKS} 
is not directly applicable. For this reason, we 
give a more self-contained discussion on this
issue here, as the topic plays a vital role in our paper.
\end{remark}

\begin{definition}
We fix a ball $B(x_0,R)\subset X$, and for each positive integer $n$ let $\Gamma_n$ be the collection of all locally rectifiable
curves in $X$ that intersect both $B(x_0,R)$ and $X\setminus B(x_0,nR)$. Following~\cite{HoloKos, HoSh} we say that
$X$ is $p$-parabolic if $\lim_{n\to\infty}\Mod_p(\Gamma_n)=0$, and $p$-hyperbolic if $\lim_{n\to\infty}\Mod_p(\Gamma_n)>0$.
As $\Gamma_n\subset\Gamma_m$ when $m<n$, the monotonicity of $p$-modulus of families of curves implies that
the above limit exists. The argument in~\cite{HoSh} shows that the property of $p$-parabolicity versus $p$-hyperbolicity
is independent of the choice of the all $B(x_0,R)$. 
\end{definition}

 Recall that in this subsection $\partial\Om$ is bounded.
If $\overline{\Om}$ is $p$-parabolic, then by~\cite{Hansevi} we know that a solution to 
the Dirichlet problem on $\Om$ exists
and is unique; however, when $\overline{\Om}$ is not $p$-parabolic, we may still have a solution in the class $D^{1,p}(\Om)$, but
the solution is not unique, see the discussion in~\cite[Section~9]{GKS}. We will see in this paper that a similar statement
holds when both $\Om$ and $\partial\Om$ are unbounded, with $\Om$ a uniform domain equipped with a doubling measure
$\mu$ supporting a $p$-Poincar\'e inequality.

\begin{remark}\label{rem:cap-of-infty}
If $\overline{\Om}$ is $p$-parabolic, then functions $v\in D^{1,p}(\Om)$ need not be well-defined at $\infty$, but 
when $\overline{\Om}$ is $p$-hyperbolic, $v(\infty)$ is well-defined. Indeed, thanks to the analogue of~\eqref{eq:ChEnergy-trans}
for minimal $p$-weak upper gradients of functions $v\in D^{1,p}(\Om)$, we know that 
$v\in N^{1,p}(\Om_{k},d_\pip,\mu_{\pip,k})$ for any
positive integer $k$; 
hence by the results of~\cite{GKS} we know that when $\overline{\Om}$ is $p$-hyperbolic, 
then $\text{Cap}_p^\pip(\{\infty\})>0$ and so $v(\infty)$ is well-defined, with 
\[
v(\infty)=\lim_{r\to0^+}\frac{1}{\mu_{\pip_{k}}(B_{\pip_{k}}(\infty, r))}\int_{B_{\pip_{k}}(\infty, r)\cap\Om}v\, d\mu_{\pip_{k}}.
\]
Recall that in the setting of~\cite{GKS}, we always have $\Mod_p^{\pip_{k}}(\Gamma)=\Mod_p(\Gamma)$ for each
family $\Gamma$ of locally rectifiable curves in $\Om$.
\end{remark}

\subsection{Uniform boundedness principle}
To conclude this background section, 
we state one of the key results that we will need in the latter half of this 
{\it the uniform boundedness principle}, also known as the
Banach--Steinhaus theorem. The version we need is the following.

\begin{theorem}[Banach--Steinhaus Theorem]\label{Thm:UnifBdd}
Let $V$ be a Banach space and $\mathcal{F}$ a family of bounded linear operators from $V$ to $\R$. If for each
$v\in V$ we have that 
\[
\sup_{L\in\mathcal{F}}|L(v)|<\infty,
\]
then 
\[
\sup_{L\in\mathcal{F}}\Vert L\Vert_{V^*}<\infty.
\]
Here, for $L\in\mathcal{F}$, the norm $\Vert L\Vert_{V^*}$ is the operator norm on $L$
given by
\[
\|L\|_{V^*}:=\sup_{v\in V\, \|v\|=1}\ |L(v)|.
\]
\end{theorem}

\section{First Step: Dirichlet problem on domains with unbounded boundary}\label{Sec.Dirichlet}

The goal of this section is to prove the following theorem regarding the Dirichlet problem on $\Om$ with 
unbounded boundary $\partial\Om$, under assumptions (H0), (H1) and (H2).

\begin{theorem}\label{thm:unbounded-Dirichlet} Denote by $\Theta$ the exponent in \eqref{eq:codim} and set $\theta=1-\Theta/p$.
Let $F\in HB^{\theta}_{p,p}(\partial\Om)$. Then
there exists a unique function $u\in D^{1,p}(\Om,\mu)$ such that $u$ is $p$-harmonic in $\Om$ and $Tu=F$ 
$\nu$-almost everywhere on $\partial\Om$, and in addition for every $v\in D^{1,p}(\Om)$ with $Tv=0$ we have
\[
\int_\Om |\nabla u|^{p-2} \nabla u \cdot \nabla v \, d\mu=0.
\]
\end{theorem}

Note again that we do not impose any condition on the test functions $v$ at $\infty$ in the above theorem.

To exhibit such a solution and to show its uniqueness, we proceed by first solving a related problem on a sequence of
domains $\Om_k$ whose  boundary is bounded, and which approximate $\Omega$ as $k\to \infty$. In 
this way for each $k\in \N$ we have a unique solution from the results 
in~\cite{GKS}. We then complete the construction by taking the limit as $k\to\infty$. 
\subsection{Dirichlet problem on $\Om_k$}
{For the convenience of the reader, now
we recall part of the discussion from Section~\ref{Sec:strategy} here.}
Fix a ball $B_0$ and, for $k$ a positive integer, consider $B_k:=2^kB_0$. Define 
$\Om_k:=\Om\cup\{\partial\Om\setminus\overline{B_k}\}$. Note that 
$\Om\subset\Om_k\subset\overline{\Om}$. We set a metric $d_k$ on $\Omega_k$ 
to be the extension of $d$ to $\overline{\Om}$ (which we continue to denote by $d$) 
restricted to $\Omega_k$; the measure $\mu_k$ on $\Omega_k$ is defined, 
for $A\subset\Omega_k$ to be $\mu_k(A)=\mu(A\cap\Om)$.

\begin{lemma}\label{L4.2}
With $\Om$ a $C_U$-uniform domain, let $A\subset\partial\Om$ be a relatively open set. Then
$\Om(A):=\Om\cup A$ is a uniform domain in its completion $\overline{\Omega(A)}=\overline{\Omega}$ with 
$(\Omega(A),d_A,\mu_A)$ a doubling metric measure space supporting a $p$-Poincar\'{e} inequality, where $d_A$ is the restriction
of the extension of $d$ to $\overline{\Om}$ back to $\Om(A)$, and $\mu_A$ is the measure given by
$\mu_A(E)=\mu(E\cap\Om)$ whenever $E\subset\Om(A)$.
\end{lemma}

\begin{proof}
	By construction, the boundary of $\Omega(A)$ is $\partial\Om\setminus A$. Note that by the definition of $\mu_A$, we 
have that $\mu_A(A)=0$. Hence, if $B_A(x,r)$ denotes the ball in $\Omega(A)$ with respect to the metric $d_A$ and 
$B(x,r)$ denotes the ball in $\Om$ with respect to the metric $d$, then they differ by $\mu_A$-measure zero.
	
	Let $x,y\in \Om(A)$. As $\Om(A)\subset\overline\Om$, there exists a curve $\gamma$ with end points $x$ and $y$ 
that is uniform in $(\Om,d)$. As $d_A$ agrees with $d$ on $\Om(A)$, we have that $d_A(x,y)=d(x,y)$ and that $\ell_A=\ell$, 
where $\ell_A$ denotes the length with respect to the metric $d_A$. Also, as $\partial{\Om(A)}\subset\partial\Om$, we have 
that $d_\Om(x)\leq d_{\Om(A)}(x)$ whenever $x\in\Om$; on the other hand, if $x\in\Om(A)\setminus\Om\subset\partial\Om$, 
then $d_{\Om}(x)=0\leq d_{\Om(A)}(x)$. In either case, this means that for any $z\in\gamma$,
\[
\min\{\ell_A(\gamma_{z,x}),\ell_A(\gamma_{z,y})\}
=\min\{\ell(\gamma_{x,z}),\ell(\gamma_{z,y})\}\le C_U\,d_{\Om}(z)\leq C_U\,d_{\Om(A)}(z)
\]
and 
\[
	\ell_A(\gamma)=\ell(\gamma)\le C_U d(x,y)=C_Ud_A(x,y),
\]
	which is to say that $\gamma$ is a uniform curve in $(\Om(A),d_A)$. Therefore, $(\Om(A),d_A)$ is uniform.
	
	Let $x\in\Om(A)$ and $r>0$. 
Then there exists $y\in\Om$ such that $d(x,y)<r/2$, and so 
\[
	B(y,r/2)\subset B_A(x,r)\cap\Omega\subset B_A(x,2r)\cap\Omega\subset B(y,3r).
\]
	From the doubling property of $\mu$, it follows that 
	\begin{align*}
	\mu_A(B_A(x,2r))=\mu(B_A(x,2r)\cap\Om)&\leq \mu(B(y,3r))\\&\leq C \mu(B(y,r/2))\leq C\,\mu(B_A(x,r)\cap\Omega). 
	\end{align*}
	Therefore, $\mu_A$ is doubling.
	
	Let $u\in D^{1,p}(\Om(A))$ with $g$ as a $p$-weak upper gradient on $\Om(A)$. As $\Om\subset\Om(A)$, 
	it follows that $u\in D^{1,p}(\Om)$ with $g$ a $p$-weak upper gradient on $\Om$. For $x\in \Om(A)$ and $r>0$, 
again choose $y\in\Om$ such that $d(x,y)<r/2$. Then, 
	$$
	B(y,r/2)\subset B_A(x,r)\cap\Om\subset B(y,2r)\subset B(y,2\lambda r)\subset B_A(x,3\lambda r)\cap\Om\,,
	$$
	and so the doubling of both $\mu$ and $\mu_A$, and the Poincar\'{e} inequality on $\Om$ imply that
	\begin{align*}
		\vint_{B_A(x,r)}\!|u-u_{B_A(x,r)}|\,d\mu_A&\approx \inf_{c\in\R}\vint_{B_A(x,r)}\!|u-c|\,d\mu_A\\
		&=\inf_{c\in\R}\vint_{B_A(x,r)\cap\Om}\!|u-c|\,d\mu\\
		&\lesssim \inf_{c\in\R}\vint_{B(y,2r)}\!|u-c|\,d\mu \lesssim r\left(\vint_{B( y, 2\lambda r)}g^p\,d\mu\right)^{1/p}\\
		&\lesssim r\left(\vint_{B_A( x, 3\lambda r)}g^p\,d\mu_A\right)^{1/p}.
	\end{align*}
	Therefore, $(\Om(A),d_A,\mu_A)$ satisfies a $p$-Poincar\'{e} inequality. 
\end{proof}

\begin{lemma}\label{lem:ext2notext}
$D^{1,p}(\Om)=D^{1,p}(\Om_k)=D^{1,p}(\overline{\Om})$. Moreover, for $w\in D^{1,p}(\Om_k)$ we have
\[
\int_\Om|\nabla w|^p\, d\mu =\int_{\Om_k}|\nabla w|^p\, d\mu=\int_{\overline{\Om}}|\nabla w|^p\, d\mu.
\]
\end{lemma}

\begin{proof}
Since $\Om\subset \overline{\Om}$, it is immediate that the restriction of a $D^{1,p}(\overline{\Om})$-function 
to $\Om$ is in $D^{1,p}(\Om)$. As $\Om\subset\Om_k\subset \overline{\Om}$, the result will follow if we can 
show that, conversely, functions in $D^{1,p}(\Om)$ have an extension to $D^{1,p}(\overline{\Om})$. 
The latter holds in view of~\cite[Proposition~7.1]{AS}.
\end{proof}

In this paper, for each positive integer $k$ we apply the above lemma with the choice of $A=\partial\Om\setminus \overline{B_k}$.
Recall that in this case, $\Om(A)=\Om_k$, and note that $\partial\Om_k=\overline{B_k}\cap\partial\Om$ is bounded.
This lemma allows us to invoke Theorem~\ref{thm:bounded-Dirichlet} to obtain a unique 
function $u_k\in D^{1,p}(\Om_k)$ such that $u_k$ is $p$-harmonic on 
$\Om_{k,\pip}=\Om_k\cup\{\infty\}$ and $Tu_k=F|_{\partial\Om_k}$, and so that for all $v\in D^{1,p}(\Om_k)$ such that 
$Tv=0$ $\nu$-a.e. on $\partial\Om_k$ we have
\begin{equation}\label{eq:weak-k}
\int_{\Om_k} |\nabla u_k|^{p-2} \nabla u_k \cdot \nabla v\, d\mu=0.
\end{equation}
Note that $EF\in D^{1,p}(\Om)$, where $E$ is the bounded linear extension operator associated with $\Om$
as constructed in~\cite[Theorem~1.1]{GS2}. Observe that the trace operator 
$T_k:N^{1,p}(\Om_{k,\pip},d_{k,\pip},\mu_{k,\pip})\to B^{\theta}_{p,p}(\partial\Om_k)$
is merely the restriction of the trace operator $T:D^{1,p}(\Om)\to HB^{\theta}_{p,p}(\partial\Om)$ to 
$\overline{B_k}\cap\partial\Om=\partial\Om_k$.

We also emphasize here that the test function $v$ in~\eqref{eq:weak-k} is not required to vanish at $\infty$
nor at $\partial\Om\setminus \overline{B_k}$. Hence, one interpretation of $u_k$ is as satisfying the mixed boundary value 
problem at $\partial\Om\cup\{\infty\}$ with Dirichlet data $F$ on $\partial\Om\cap \overline{B_k}$ and zero
Neumann data on $(\partial\Om\setminus\overline{B_k})\cup\{\infty\}$.

\subsection{Dirichlet  problem on $\Om$}

{Recall the notion of trace $T:D^{1,p}(\Om)\to HB^{1-\Theta/p}_{p,p}(\partial\Om)$ 
and the related extension operator $E$
from Subsection~\ref{s:trace}.
Let $F$ be as in the statement of Theorem~\ref{thm:unbounded-Dirichlet} above.}
Given  Lemma \ref{L4.2}, we are allowed to use the fact that $EF\in D^{1,p}(\Om_k)$ with $TEF=F$ $\nu$-a.e.~on
$\partial\Om_k$, and hence by the energy minimization property of solutions, we see that
\begin{equation}\label{eq:Dir-bound}
\int_\Om|\nabla u_k|^p\, d\mu=\int_{\Om_k}|\nabla u_k|^p\, d\mu \le \int_{\Om_k}|\nabla EF|^p\, d\mu
   =\int_\Om|\nabla EF|^p\, d\mu.
\end{equation}
Observe that $EF$ is independent of $k$, and is as constructed in~\cite{GS2}. Therefore, as $EF\in D^{1,p}(\Om)$,
the right-handmost side of~\eqref{eq:Dir-bound} is independent of $k$.

In this subsection we give a proof of the existence component claimed in Theorem~\ref{thm:unbounded-Dirichlet}. 

For each positive integer $k$ we set $u_k$ to be the solution constructed in the previous subsection. 
Then from~\eqref{eq:Dir-bound} we have that the sequence $(u_k)_k$ is bounded in $D^{1,p}(\Om)$ with 
respect to the seminorm of $D^{1,p}(\Om)$.

Note that for each positive integer $k$, the set $\Om_{k,\pip}=\Om_k\cup\{\infty\}$.

\begin{lemma}\label{lem:bdd-solutions}
Fix $k_0\in\N$ with $k_0\ge 2$. Then there is a constant $C=C(k_0)>0$ such that for each positive integer $k\ge k_0$ we have that 
\[
\int_\Om|u_k|^p\, d\mu_{k_0,\pip}\le C\left[\int_\Om|\nabla_{k_0,\pip} EF|^p\, d\mu_{k_0,\pip}
+\int_\Om |EF|^p\, d\mu_{k_0,\pip}\right].
\]
It follows that $(u_k)_{k\ge k_0}$ is a bounded sequence in 
$N^{1,p}(\Om_{k_0,\pip},d_{k_0,\pip}, \mu_{k_0,\pip})$.
\end{lemma}

\begin{proof}
Note that $T(u_k-EF)=0$ on $\partial\Om_{k_0}=\overline{2^{k_0}B}\cap\partial\Om$ $\nu$-a.e.~when $k\ge k_0$ since 
$\partial\Om_{k_0}\subset \partial\Om_k$. Hence by the Maz'ya capacitary inequality~\cite[Theorem~6.21]{BBbook} 
applied to the bounded space $(\overline{\Om_{k_0,\pip}}^\pip,d_{k_0,\pip},\mu_{k_0,\pip})$,
we see that
\[
\int_\Om|u_k-EF|^p\, d\mu_{\pip_{k_0}}\le C\, \int_\Om|\nabla (u_k-EF)|^p\, d\mu.
\]
We point out that the constant $C$ in the above inequality depends on $k_0$.
Here we used the fact that the variational $p$-capacity of any subset of $B_{k_0}$ is positive if and only if
its Sobolev $p$-capacity in $(\overline{\Om_{k_0,\pip}}^\pip,d_{k_0,\pip},\mu_{k_0,\pip})$
is positive, see~\cite[Lemma~6.15]{BBbook}; and in addition, we know 
from~\cite[Proposition~8.2]{GKS} that the Sobolev $p$-capacity of $B_{k_0-1}\cap\partial\Om$ is positive.
Hence the Maz'ya inequality is meaningful here.
It follows that
\begin{align*}
\int_\Om|u_k|^p\, d\mu_{k_0,\pip}
  &\le C\left[\int_\Om|u_k-EF|^p\, d\mu_{k_0,\pip}+\int_\Om|EF|^p\, d\mu_{k_0,\pip}\right]\\
  &\le C\left[\int_\Om|\nabla (u_k-EF)|^p\, d\mu+\int_\Om|EF|^p\, d\mu_{k_0,\pip}\right]\\
  &\le C\left[\int_\Om|\nabla u_k|^p\, d\mu+\int_\Om|\nabla EF|^p\, d\mu+\int_\Om|EF|^p\, d\mu_{k_0,\pip}\right]\\
  &\le C\left[\int_\Om|\nabla EF|^p\, d\mu+\int_\Om|EF|^p\, d\mu_{k_0,\pip}\right]\\
  &\le C\left[\int_\Om|\nabla_{k_0,\pip}EF|^p\, d\mu_{k_0,\pip}+\int_\Om|EF|^p\, d\mu_{k_0,\pip}\right],
\end{align*}
where we used~\eqref{eq:Dir-bound} in the penultimate step.
\end{proof}

\begin{proposition}\label{prop:exists}
For each $F\in HB^{\theta}_{p,p}(\partial\Om)$ there is a function $u\in D^{1,p}(\Om)$ such that 
\begin{align}\label{eq:Problem-Solution}
Tu =F \  \text{ on }\ \partial\Om,& \notag\\
\int_\Om |\nabla u|^{p-2}\  \nabla u \cdot \nabla v \, d\mu&=0 
\end{align}
for each $v\in D^{1,p}(\Om)$ with trace $T v=0$ on $\partial\Om$.
\end{proposition}

\begin{proof}
Lemma~\ref{lem:bdd-solutions} tells us that the sequence $(u_k)_{k\ge k_0}$ is a bounded sequence in
$N^{1,p}(\overline{\Om_{k_0,\pip}}^\pip)$, and so by the reflexivity of $N^{1,p}(\overline{\Om_{k_0,\pip}}^\pip)$ 
there is a function $u_{\infty,k_0}\in N^{1,p}(\overline{\Om_{k_0,\pip}}^\pip)$ and a subsequence $(u_{m_j})_{j\in\N}$
that converges to $u_{\infty, k_0}$ weakly in $N^{1,p}(\overline{\Om_{k_0,\pip}}^\pip)$. This in particular means
that $(u_{m_j})_{j\in\N}$ converges to $u_{\infty, k_0}$ weakly in $L^p(\overline{\Om_{k_0,\pip}}^\pip)$,
$(\nabla u_{m_j})_{j\in\N}$ converges weakly to $\nabla u_{\infty, k_0}$ in $L^p(\overline{\Om_{k_0,\pip}}^\pip;\R^N)$, 
and by
Mazur's lemma, there is a convex-combination subsequence 
$\widetilde{u_n}=\sum_{j=n}^{N(n,k_0)}\lambda_{j,n,k_0}\, u_{k_j}$ with $\lambda_{j,n,k_0}\ge 0$ and 
$\sum_{j=n}^{N(n,k_0)}\lambda_{j,n,k_0}=1$, such that $\widetilde{u_n}$ converges to 
$u_{\infty, k_0}$ in $N^{1,p}(\Om_{k_0,\pip})$.

By~\cite[Corollary~4.8]{Sh}, we also ensure that the subsequence $u_{k_m}\to u_{\infty,k_0}$ locally uniformly in
$\Om$ and that $u_{\infty, k_0}$ is Cheeger $p$-harmonic in the bounded (in the metric $d_{k_0,\pip}$)
domain $\Om\subset\Om_{\pip_{k_0}}$. By the linearity of the trace operator $T$, and by the fact that 
$Tu_k=F$ on $\partial\Om\cap \overline{2^kB_0}$, it follows that $Tu_{\infty,k_0}=F$ on $\partial\Om$ (recall that
$\partial\Om=\bigcup_k\partial\Om\cap \overline{2^kB_0}$). Moreover, as $u_{\infty,k_0}\in N^{1,p}(\Om_{k_0,\pip})$,
necessarily $u_{\infty, k_0}\in D^{1,p}(\Om)$. By the $p$-harmonicity of this limit function on $\Om$, it follows that whenever
$v\in D^{1,p}(\Om)$ with $Tv=0$ on $\partial\Om$ and $v(\infty)=0$, we have
\[
\int_{\Om_k} |\nabla u_{\infty,k_0}|^{p-2}\, \nabla u_{\infty,k_0} \cdot \nabla v \, d\mu=0.
\]
However, we also wish to show that the above equation holds even when $v(\infty)\ne 0$. To this end, observe that
when $k, m$ are positive integers with $k>m$, we know that $u_m$ solves a Dirichlet problem on $\Om_{m,\pip}$
with boundary data $F$ on $\overline{2^mB_0}\cap\partial\Om$, and $Tu_k=F$ on 
$\overline{2^mB_0}\cap\partial\Om$. It follows by the energy
minimization property of $u_m$ that
\[
\int_\Om|\nabla u_m|^p\, d\mu\le \int_\Om|\nabla u_k|^p\, d\mu.
\]
Thus $(\Vert \nabla u_m\Vert_{L^p(\overline{\Om_{k_0,\pip}}^\pip;\R^N)}$ is monotone increasing, and is bounded 
by~\eqref{eq:Dir-bound}. 
Moreover, as the trace of $u_{\infty,k_0}$ equals $F$ in $\overline{2^{m_j}B_0}\cap\partial\Om$, it follows that
\[
\int_\Om|\nabla u_{m_j}|^p\, d\mu\le \int_\Om|\nabla u_{\infty,k_0}|^p\, d\mu.
\]
Thus we have that 
\[
\lim_{j\to\infty}\Vert \nabla u_{m_j}\Vert_{L^p(\overline{\Om_{k_0,\pip}}^\pip;\R^N)}
  =\Vert \nabla u_{\infty,k_0}\Vert_{L^p(\overline{\Om_{k_0,\pip}}^\pip;\R^N)},
\]
and as in addition we have that $(\nabla u_{m_j})_j$ converges weakly in $L^p(\overline{\Om_{k_0,\pip}}^\pip;\R^N)$ to
$\nabla u_{\infty, k_0}$, it follows that $\nabla u_{m_j}\to \nabla u_{\infty,k_0}$ strongly in $L^p(\overline{\Om_{k_0,\pip}}^\pip;\R^N)$ 
see for example~\cite[Proposition~2.4.17]{HKST} or~\cite[Proposition 3.32]{B}.
Note that
if $v\in N^{1,p}(\overline{\Om_{k_0,\pip}}^\pip)=D^{1,p}(\overline{\Om})$ such that $Tv=0$ on $\partial\Om$, 
then for each $k\in\N$ we have that 
\begin{equation}\label{eq:harm-infty}
\int_{\Om_{k_0,\pip}} |\nabla u_k|^{p-2}\,\nabla u_k\cdot \nabla v \,d\mu=0.
\end{equation}
As $(\nabla u_{m_j})_j$ converges strongly in $L^p(\overline{\Om_{k_0,\pip}}^\pip;\R^N)$ to $\nabla u_{\infty, k_0}$,
by passing to a further subsequence if necessary, we can  ensure that this convergence is also pointwise a.e.~in $\Om$.

Moreover, $(|\nabla u_{m_j}|^{p-2}\nabla u_{m_j})_j$ is bounded in $L^{p/(p-1)}(\overline{\Om_{k_0,\pip}}^\pip;\R^N)$,
and hence, by passing to a further subsequence if necessary, we can also ensure that this sequence converges weakly; and
then, thanks to Mazur's lemma together with the above pointwise convergence, we conclude that the weak limit
is $|\nabla u|^{p-2}\nabla u$.
Hence,
\[
0=\lim_j \int_{\Om_{k_0,\pip}} |\nabla u_{m_j}|^{p-2}\,\nabla u_{m_j}\cdot \nabla v \,d\mu 
  = \int_{\Om_{k_0,\pip}} |\nabla u_{\infty, k_0}|^{p-2}\,\nabla u_{\infty, k_0}\cdot\nabla v \,d\mu,
\]
that is,~\eqref{eq:harm-infty} holds also when $v(\infty)\ne 0$.
It follows that $u_{\infty,k_0}$ is $p$-harmonic in $\Om\cup\{\infty\}$, that is, it satisfies~\eqref{eq:harm-infty} for
each $v\in D^{1,p}(\Om)$ with $Tv=0$ on $\partial\Om$. 

As $Tu_k=F$ on $\overline{2^kB_0}\cap\partial\Om$,
it follows that $u_k-F\in N^{1,p}_0(\Om_{k_0,\pip})$. Hence as $u_{\infty,k_0}$ is the strong limit of
the convex combination of the subsequence $(u_{m_j})_{j\in\N}$, and as the trace operator is linear and bounded,
it follows that $Tu_{\infty, k_0}=F$ on $\partial\Om$. 
Thus {$u_{\infty,k_0}$ }solves the boundary value problem
\begin{align*}
Tu_{\infty, k_0} =F \  \text{ on }\ \partial\Om,&\\
\int_\Om |\nabla u_{\infty,k_0}|^{p-2}\, \nabla u_{\infty,k_0}\cdot \nabla v\, d\mu&=0 
\end{align*}
for each $v\in D^{1,p}(\Om)$ with $Tv=0$ on $\partial\Om$. That is, the choice of $u=u_{\infty,k_0}$ solves the 
problem~\eqref{eq:Problem-Solution}.
\end{proof}

\subsection{Uniqueness of the solution}

In this subsection we are interested in establishing the uniqueness of the solution to the problem~\eqref{eq:Problem-Solution}
described in the previous section; in particular, we would like to know that different choices of $k_0$ do not result in different
solutions to the problem~\eqref{eq:Problem-Solution}.

So suppose that $u,v$ are two functions in $D^{1,p}(\Om)$ that solve the problem~\eqref{eq:Problem-Solution}. Then
$u-v\in D^{1,p}(\Om)$ with trace $T(u-v)=0$ on $\partial\Om$. Hence, applying~\eqref{eq:Problem-Solution} to
both $u$ and to $v$, with test function $u-v$, gives
\begin{align*}
\int_\Om|\nabla u|^{p-2} \nabla u\cdot\nabla (u-v)\, d\mu&=0,\\
\int_\Om|\nabla v|^{p-2} \nabla v\cdot\nabla (u-v)\, d\mu&=0.
\end{align*}
It follows that we have 
\[
\int_\Om \left[|\nabla u|^{p-2}\nabla u-|\nabla v|^{p-2}\nabla v\right]\cdot \nabla (u-v) \, d\mu=0.
\]
As $\left[|\nabla u|^{p-2}\nabla u-|\nabla v|^{p-2}\nabla v\right]\cdot \nabla (u-v)\ge 0$ when $p>1$, it follows that
$\left[|\nabla u|^{p-2}\nabla u-|\nabla v|^{p-2}\nabla v\right]\cdot \nabla (u-v)=0$ $\mu$-a.e.~in $\Om$
(see, for example,~\cite[Chapter~3, (3.6)]{HKM}).
Again because
$p>1$, we see that $\nabla u=\nabla v$ $\mu$-a.e.~in $\Om$, that is, $\nabla (u-v)=0$. As $|\nabla (u-v)|\approx g_{u-v}$
where $g_{u-v}$ is the minimal $p$-weak upper gradient of $u-v$ in $\Om$, it follows that $g_{u-v}=0$ $\mu$-a.e.~in $\Om$ as
well. Thus $u-v$ is constant in $\Om$. As the trace $T(u-v)=0$ in $\partial\Om$, it follows that this constant is zero; that is,
$u=v$ $\mu$-a.e.~in $\Om$, and hence, taking the continuous representative of $u$ and of $v$ (and such representatives
exist by the H\"older continuity property established in~\cite{KiSh}), it follows that $u=v$ everywhere in $\Om$.\qed

\section{Second Step: Neumann boundary-value problem on $\Om$ for compactly supported Neumann data 
$f\in L^{p'}(\partial\Om)$}

 We finally turn to the solution of the Neumann problem. As mentioned earlier, this will be achieved in two steps. In this section, 
 we will only consider compactly supported data and, in the next section, we will develop an approximation scheme to solve the 
 problem for data with arbitrary support.

\begin{proposition}\label{prop:compact-Neumann}
Let $f\in L^{p'}(\partial\Om)$ such that the support of $f$ is contained in $\overline{B_0}\cap\partial\Om$ and 
$\int_{\partial\Om}f\, d\nu=0$. Then there is a function $u\in D^{1,p}(\Om)$ such that whenever $v\in D^{1,p}(\Om)$, we have
\begin{equation}\label{eq:Neum-prob}
\int_\Om|\nabla u|^{p-2}\, \nabla u \cdot \nabla v \, d\mu=\int_{\partial\Om} Tv\, f\, d\nu.
\end{equation}
Moreover, if $u^*\in D^{1,p}(\Om)$ also satisfies the above equation, then $u-u^*$ is constant on $\Om$.
\end{proposition}

\begin{proof}
We now have $f\in L^{p'}(\partial\Om,d\nu)$ with integral $\int_{\partial\Om}f\, d\nu=0$ and support inside $B_0$. Then consider 
$\Om^*=\Om_2$ in the previous notation. Note that then 
$\Om^*_\pip=\Om_{2,\pip}$ has two conformally equivalent identities: one as a metric measure space with the metric coming 
from the restriction of the metric $d$
and the measure $\mu$; and, the second with the metric $d_{2,\pip}$ and the measure $\mu_{2,\pip}$. It is this latter 
identity that we consider here, and $(\Om^*_\pip,d_\pip,\mu_\pip)$ stands in for the metric measure space
$(\Om_{2,\pip},d_{2,\pip},\mu_{2,\pip})$.

Recalling the definition of $|\nabla_\pip u|_\pip$ from \eqref{D:nabla phi}, we also define the Neumann energies $I, I_\pip$ to be
\begin{align*}
I(u)&:=\int_\Om|\nabla u|^p\, d\mu-p\, \int_{\partial\Om}f\, Tu\, d\nu,\\
I_\pip(u)&:=\int_{\Om^*_\pip}|\nabla_\pip u|_\pip^p\, d\mu_\pip-p\, \int_{\partial\Om^*_\pip}f\, Tu\, d\nu,
\end{align*}
for $u\in D^{1,p}(\Om)=N^{1,p}(\overline{\Om^*_\pip}^\pip)$, which follows from Lemmas~\ref{lem:nopiptopip} and~\ref{lem:ext2notext}.
 By the choice of $\mu_\pip$ and the differential structure $\nabla_\pip$, we know that
\[
\int_\Om|\nabla u|^p\, d\mu=\int_{\Om^*_\pip}|\nabla_\pip u|_\pip^p\, d\mu_\pip.
\]
From the construction of $d_\pip$ and by the fact that the support of $f$ is contained in $\overline{B_0}$, we know that 
\[
\int_{\partial\Om}f\, Tu\, d\nu=\int_{\partial\Om^*_\pip}f\, Tu\, d\nu.
\]
It follows that for each $u\in D^{1,p}(\Om)$, necessarily $I(u)=I_\pip(u)$. However, the metric measure space
$(\Om^*_\pip,d_\pip,\mu_\pip)$ is a \emph{bounded} uniform domain that satisfies the hypotheses considered in
the paper~\cite{MaSh}, and so we can appeal to the results therein to obtain a minimizer $u\in D^{1,p}(\Om)$ 
of the energy $I(u)$. Hence, by considering the Euler-Lagrange formulation of the energy minimization, 
the first part of the proposition follows.

Now suppose that $u^*$ is another solution to the equation~\eqref{eq:Neum-prob}. To show that $u-u^*$ is constant
$\mu$-a.e.~in $\Om$, we argue as follows. 
Since, from the above argument, we know that $I(u^*)=I_\pip(u^*)$, we also know that $u^*$ is a solution 
to the Neumann boundary value problem on $(\Om^*_\pip, d_\pip,\mu_\pip)$ with boundary data $f$, just as $u$ is.
From~\cite[Lemma~4.5]{MaSh} we know that 
$|\nabla_\pip u|_\pip=|\nabla_\pip u^*|_\pip$ $\mu_\pip$-a.e.~in $\Om$ and 
$\int_{\partial\Om}Tu\, f\, d\nu=\int_{\partial\Om}Tu^*\, f\, d\nu$. Moreover, we have 
$I_\pip(u)=I_\pip(u^*)$; whence we observe that, via the Euler-Lagrange equation for the minimization of $I_\pip$ that
for each $w\in D^{1,p}(\Om)=N^{1,p}(\overline{\Om_\pip^*})$,
\[
\int_\Om |\nabla_\pip u|_\pip^{p-2}\langle \nabla_\pip u,\nabla_\pip w\rangle_\pip\,d\mu_\pip
=
\int_\Om |\nabla_\pip u^*|_\pip^{p-2}\langle \nabla_\pip u^*,\nabla_\pip w\rangle_\pip\,d\mu_\pip.
\]
Choosing $w=u-u^*$, we see that
\[
\int_\Om\langle |\nabla_\pip u|_\pip^{p-2}\nabla_\pip u - |\nabla_\pip u^*|_\pip^{p-2}\nabla_\pip u^*,\, 
\nabla_\pip(u-u^*)\rangle_\pip\, d\mu_\pip =0.
\]
By convexity of $L^p$-norm, this holds if and only if $\nabla_\pip u=\nabla_\pip u^*$ $\mu_\pip$-a.e.~in $\Om$.
That is, $\nabla_\pip(u-u^*)=0$ $\mu_\pip$-a.e.~in $\Om$. As $(\Om,d_\pip,\mu_\pip)$ supports a
$p$-Poincar\'e inequality and is bounded, it follows that $u-u^*$ is constant $\mu_\pip$-a.e.~in $\Om$,
and so is constant $\mu$-a.e.~in $\Om$. As $u, u^*$ are continuous on $\Om$, it follows that
$u-u^*$ is constant in $\Om$, completing the proof of uniqueness.
\end{proof}

\begin{corollary}\label{cor-bddOp}
If $f\in L^{p^\prime}(\partial\Om)$ has compact support with $\int_{\partial\Om}f\, d\nu=0$, then the map 
$L_f:HB^\theta_{p,p}(\partial\Om)\to\R$ given by $L_f(v)=\int_{\partial\Om}f\, v\, d\nu$ is a bounded linear operator
with operator norm $\|L_f\|\approx \|\nabla u_f\|_{L^p(\Om)}^{1-1/p}$, where $u_f$ is the function
given by Proposition~\ref{prop:compact-Neumann} above.
\end{corollary}

Indeed, the conclusion of the above corollary holds even non-compactly supported functions $f$ as long as
there is a function $u\in D^{1,p}(\Om,\mu)$ that satisfies~\eqref{eq:Neum-prob} for each $v\in D^{1,p}(\Om,\mu)$.

\begin{proof}
That $L_f$ is linear is clear. To see the boundedness of $L_f$, note that by Proposition~\ref{prop:compact-Neumann} above,
there is a function $u_f\in D^{1,p}(\Om)$ solving the problem~\eqref{eq:Neum-prob}. Hence, for each 
$v\in HB^\theta_{p,p}(\partial\Om)$, choosing the extension $Ev\in D^{1,p}(\Om)$ of $v$, we have
\begin{align*}
\left|\int_{\partial\Om}f\, v\, d\nu\right|&=\left|\int_\Om|\nabla u_f|^{p-2} \nabla u_f\cdot\nabla Ev\, d\mu\right|\\
  &\le \left(\int_\Om|\nabla u_f|^p\, d\mu\right)^{1-\tfrac1p}\, \left(\int_\Om|\nabla Ev|^p\, d\mu\right)^{\tfrac1p}\\
  &\lesssim \left(\int_\Om|\nabla u_f|^p\, d\mu\right)^{1-\tfrac1p}\, \|v\|_{HB^\theta_{p,p}(\partial\Om)}.
\end{align*}
Thus, $L_f$ is bounded with operator norm
\[
\|L_f\|\lesssim \left(\int_\Om|\nabla u_f|^p\, d\mu\right)^{1-\tfrac1p}.
\]

On the other hand, choosing $v=Tu_f$ in \eqref{eq:Neum-prob}, we also know that 
\begin{align*}
\int_\Om|\nabla u_f|^p\, d\mu=\int_{\partial\Om}f\, Tu_f\, d\nu&\le \|L_f\|\, \|Tu_f\|_{HB^{1-\theta/p}_{p,p}(\partial\Om)}\\
  &\lesssim \|L_f\|\, \left(\int_\Om |\nabla u_f|^p\, d\mu\right)^{1/p},
\end{align*}
from which we obtain the final claim of the corollary.
\end{proof}

\section{Third Step: Neumann boundary-value problem on ${\Om}$ for non compactly supported Neumann 
data}

In this section, we solve the Neumann problem for more general, not necessarily compactly supported, boundary data.  
Let us assume that $\partial\Omega$ is uniformly perfect, and we will work with data $f$ satisfying the following hypotheses:

\bigskip

\begin{itemize}
\item $f\in L^{p'}(\partial \Omega, \nu)$ satisfies the integrability condition \eqref{intro-weight-h}, i.e.
$$
\int_Z|f(x)|^{p^\prime}\, d(x,x_0)^{\theta p^\prime}\, \nu(B(x_0,d(x_0,x)))^{p^\prime/p}\, d\nu(x)<\infty,
$$
for some $x_0\in \partial \Omega$ (and consequently {\it for all} $x_0\in \partial \Omega$).
\item $\int_{\partial \Omega}f d\nu=0$; 
\end{itemize}

\begin{remark} 
Obviously, in the case when $\partial \Omega$ is bounded, the integrability condition~\eqref{intro-weight} 
is always satisfied when $f\in L^{p'}(\partial \Omega, \nu)$, thus implying the existence theorem in \cite{CKKSS}.
\end{remark}

The next result will allow us to use the uniform boundedness principle.

\begin{lemma}\label{L:decay-implies-3'} Under the hypotheses above, if we set 
\begin{equation}\label{eq:fk-def1}
f_k:=f\chi_{B_k}-\frac{1}{\nu(B_0)}\left[\int_{B_k}f\, d\nu\right]\chi_{B_0},
\end{equation}
then 
\begin{equation}\label{L3}
\sup_k\, \bigg\vert \int_{\partial\Om} f_k\, h\, d\nu\bigg\vert<\infty
\end{equation}
for each $h\in HB^\theta_{p,p}(\partial\Om, \nu)$.
\end{lemma}

\begin{proof} 
Let $f_k$ be as in~\eqref{eq:fk-def1}. We wish to check that Condition \eqref{L3} holds for the family $\{f_k\}_{k\in\N}$.
Let $v\in HB^\theta_{p,p}(\partial\Om)$, and invoke Lemma \ref{panacea-bis} to obtain

\begin{equation}
|L_{f_k}(v)|\le \left(\int_{B_{2k}}|f_k|^{p^\prime}J(x,y)^{p^\prime}\, d\nu(x)\right)^{1/p^\prime}\, 
\|v\|_{HB^\theta_{p,p}(\partial\Om)},
\end{equation}
for $J$ as in \eqref{weight-v-bis} and for some $y\in \partial \Omega$. 
This calculation shows that a sufficient condition for   \eqref{L3} to hold is the following:
\begin{equation}\label{eq:fk-Jv}
\sup_{k\in\N}\int_{\partial\Om}|f_k|^{p^\prime}J(x,y)^{p^\prime}\, d\nu(x)<\infty.
\end{equation}  
Next we show that this condition is implied by \eqref{intro-weight-h}. Recalling the definition of $f_k$ one has

\begin{equation}\label{bound-final}
\int_{\partial\Om}|f_k|^{p^\prime}\, {J(x,y)}^{p^\prime}\, d\nu
\lesssim \int_{\partial\Om}|f|^{p^\prime}{J}(x,y)^{p^\prime}\, d\nu
  + \frac{1}{\nu(B_0)^{p^\prime}}\left|\int_{B_k} f\, d\nu\right|^{p^\prime}\, 
          \int_{B_0}{J(x,y)}^{p^\prime}\, d\nu.
\end{equation}
Note that by virtue of \eqref{intro-weight-h} and because 
$$\lim_{k\to \infty} \int_{B_k}f\, d\nu=0,$$
the right-hand side of \eqref{bound-final}  can be bounded \emph{independently of} the integer $k\in\N$. 
\end{proof}

Next we show that not only $f_k\to f$ in $L^{p'}(\partial \Om, \nu)$, but indeed the convergence also holds in the weighted space
$L^{p'}(\partial \Omega, \nu_J)$. 

\begin{lemma}\label{L:weighted L^p convergence}
In the notation and the hypotheses above, one has that
\begin{equation}\label{weighted convergence}
\lim_{k\to \infty} \int_{\partial \Omega} |f(y) - f_k(y)|^{p'}  d(x,x_0)^{\theta p^\prime}\, \nu(B(x_0,d(x_0,x)))^{p^\prime/p}\, d\nu(x)=0.
\end{equation}
\end{lemma}
\begin{proof}
By a direct computation we have
\begin{multline}
\int_{\partial \Omega} |f(y) - f_k(y)|^{p'}  d(x,x_0)^{\theta p^\prime}\, \nu(B(x_0,d(x_0,x)))^{p^\prime/p}\, d\nu(x) \\ \le
\int_{B_k^C\cap \partial \Omega} |f(y)|^{p'} d(x,x_0)^{\theta p^\prime}\, \nu(B(x_0,d(x_0,x)))^{p^\prime/p}\, d\nu(x) 
\\+
C \nu(B_0)^{1-p'} \Bigg(\int_{B_k} f(y) d\nu(y) \Bigg)^{p'}
\end{multline}
The first term on the right-hand side converges to zero because of the hypothesis~\eqref{intro-weight-h}. 
The second term vanishes as $k\to \infty$ in view of the fact that $f$ has zero average.
\end{proof}
\begin{corollary}\label{L: f_k Cauchy}
The sequence $\{f_k\}$ is Cauchy in $L^{p'}(\partial \Omega, \nu_J)$.
\end{corollary}

Note that  in view of the assumptions on $f$, one has that $\text{supp } f_k \subset B_{k+1}$ and 
$ \int_{\partial \Omega} f_k d\nu =0$. This allows us to invoke  Proposition \ref{prop:compact-Neumann} and deduce that there 
exist functions $u_k\in D^{1,p}(\Om)$ such that for any $v\in D^{1,p}(\Om)$, 
\begin{equation}\label{eq:Neum-prob-k}
\int_\Om|\nabla u_k|^{p-2} \nabla u_k\cdot\nabla v\, d\mu=\int_{\partial\Om} Tv\, f_k\, d\nu.
\end{equation}
We also note that in view of Corollary \ref{cor-bddOp} one has that the operator norm of the functional 
$L_f:HB^\theta_{p,p}(\partial\Om)\to\R$ satisfies 
$$\|L_{f_k}\|\approx \|\nabla u_k\|_{L^p(\Om)}^{1-1/p}$$
We want to show that we can take the limit of both sides of \eqref{eq:Neum-prob-k}, to obtain a solution of the problem for data $f$. 

\begin{lemma}\label{lemma: Lpboundf_k} 
Under the hypotheses  above, there exists a constant $C>0$, 
independent of $k$,  
such that for each $k\in \N$ 
\begin{equation}\label{eq:bndu_k}
\int_{\Omega} |\nabla u_k|^p d\mu \le C.
\end{equation}
\end{lemma}

\begin{proof}
Set $v=u_k$ in \eqref{eq:Neum-prob-k}. It follows from \cite{Sha23} and \cite{GS2} that the homogeneous 
Besov space $HB^\theta_{p,p}(\partial\Om)$ is complete and so, by 
the uniform boundedness principle (see Theorem~\ref{Thm:UnifBdd})
and the bound \eqref{L3} above, there exists $C>0$, independent on $k$, such that 
one has
\[
\int_\Omega |\nabla u_k|^p d\mu = \int_{\partial \Omega} f_k Tu_k d\nu \le C\| u_k\|_{HB_{p,p}^\theta}.
\]
Recalling the continuity of the trace operator in $D^{1,p}{(\Omega)}$ from Theorem~\ref{thm-trace-local-est},  there 
exists a constant $C_1>0$ depending only on the structural constants such that
\[
\| u_k\|_{HB_{p,p}^\theta {(\partial\Omega,\nu)}} \le C_1 \bigg(\int_\Omega |\nabla u_k|^p d\mu\bigg)^{\frac{1}{p}},
\]
the proof is now complete.
\end{proof}

The previous lemma, and weak compactness of $L^p(\Om;\R^N)$, 
yield the existence of a vector valued function 
\[
\Phi \in L^p(\Omega; \R^N)
\] 
such that 
\[
\nabla u_k \to \Phi
\] 
weakly in $L^p(\Omega; R^N)$. In the next lemma we show that this convergence is in fact a strong $L^p$ convergence.

\begin{lemma}\label{L:Cauchy}
The sequence $\{\nabla u_k\}$ is  Cauchy  in $L^p(\Omega, \mu)$.
\end{lemma}

\begin{proof}
Consider two solutions $u_m, u_n$ of the problems
\begin{align*}\int_\Om|\nabla u_n|^{p-2}\, \nabla u_n\cdot \nabla v\, d\mu&=\int_{\partial\Om} Tv\, f_n\, d\nu  \quad\text{ and }\\
\int_\Om|\nabla u_m|^{p-2}\, \nabla u_m \cdot \nabla v\, d\mu&=\int_{\partial\Om} Tv\, f_m\, d\nu
\end{align*}
for any $v\in D^{1,p}(\Omega)$.
Choose $v=u_m-u_n\in D^{1,p}(\Omega)$ and subtract one equation from the other to obtain
\begin{multline}
\int_{\Omega} (|\nabla u_m|^{p-2} \nabla u_m - |\nabla u_n|^{p-2} \nabla u_n) \cdot (\nabla u_m- \nabla u_n) \,d\mu \\
= \int_{\partial \Om} (f_m - f_n) (y) (u_m-u_n)(y)\, d\nu(y).
\end{multline}
Next, we use the fact that $f_m, f_n$ have zero average, then invoke Lemma \ref{panacea-bis} and the trace theorem, 
Theorem~\ref{thm-trace-local-est}, to arrive at

  \begin{multline}\label{eqn:617}
  \int_{\Omega} (|\nabla u_m|^{p-2} \nabla u_m - |\nabla u_n|^{p-2} \nabla u_n) \cdot (\nabla u_m- \nabla u_n) \,d\mu 
\\ \le  C  \Bigg( \int_{\partial \Om}|f_m(y)-f_n(y)|^{p'}J(x_0,y) \,d \nu(y)\Bigg)^{\frac{1}{p'}}  ||u_m-u_n||_{B^{\theta}_{p,p}(\partial \Omega)} 
 \\ \le C \Bigg( \int_{\partial \Om}|f_m(y)-f_n(y)|^{p'}J(x_0,y) \,d \nu(y)\Bigg)^{\frac{1}{p'}}  \| \nabla u_m -\nabla u_n\|_{L^p(\Omega)},
\end{multline}
for some $x_0\in \partial \Omega$, and where  $J(x,y)$ is as in \eqref{eq-panacea-bis}.   
Next we apply the monotonicity property: there exists a constant $C>0$  depending on $p$, such that for every $z,w\in \R^N$ one has
\begin{equation}\label{eqn:monotonicity}
(|z|^{p-2}z-|w|^{p-2}w)\cdot(z-w)\ge
\begin{cases}
 C|z-w|^p, & p\ge 2\\
 C (|z|+|w|)^{p-2}|z-w|^2, & p\le 2.
\end{cases}
\end{equation}
From the latter and from \eqref{eqn:617}, in the case $p\ge 2$ we obtain
\begin{equation}
\int_{\Omega} |\nabla u_m - \nabla u_n |^p d\mu  \le C \int_{\partial \Omega} |f_m - f_n|^{p'}(y) J(x_0,y) d\nu(y),
\end{equation}
and the conclusion follows from Lemma \ref{L:weighted L^p convergence}. 

If $1<p\le 2$, we apply H\"older inequality, \eqref{eqn:monotonicity}, and \eqref{eqn:617} to show that
\begin{multline*}
\|\nabla u_m - \nabla u_n\|_{L^p(\Omega)}^p\\ \le \Bigg(\int_\Om(|\nabla u_m|+|\nabla u_n|)^p\Bigg)^{\frac{2-p}{2}} \!\!\!\Bigg(\int_\Om
|\nabla u_m - \nabla u_n|^2  (|\nabla u_m|+|\nabla u_n|)^{p-2} d\mu
\Bigg)^{\frac{p}{2}}
\\
\le {C} \sup_k \| \nabla u_k\|_{L^p(\Omega)}^{\frac{p(2-p)}{2}} \Bigg(  \int_{\Omega} (|\nabla u_m|^{p-2} \nabla u_m -| \nabla u_n|^{p-2} \nabla u_n) \cdot (\nabla u_m- \nabla u_n) d\mu\Bigg)^{\frac{p}{2}} \\
\le C \sup_k \| \nabla u_k\|_{L^p(\Omega)}^{\frac{p(2-p)}{2}} \Bigg( \int_{\partial \Om}|f_m(y)-f_n(y)|^{p'}J(x_0,y) d \nu(y)\Bigg)^{\frac{p}{2p'}}  \| \nabla u_m -\nabla u_n\|_{L^p(\Omega)}^{\frac{p}{2}}.
\end{multline*}
Dividing both sides by $ \| \nabla u_m -\nabla u_n\|_{L^p(\Omega)}^{\frac{p}{2}}$, one has
\begin{equation}
\|\nabla u_m - \nabla u_n\|_{L^p(\Omega)}^p \le C \sup_k \| \nabla u_k\|_{L^p(\Omega)}^{p(2-p)}
 \Bigg( \int_{\partial \Om}|f_m(y)-f_n(y)|^{p'}J(x_0,y) d \nu(y)\Bigg)^{\frac{p}{p'}}.
\end{equation}
The conclusion now follows from Lemma \ref{lemma: Lpboundf_k} and Corollary \ref{L: f_k Cauchy}.
\end{proof}

In order to invoke more compactness in the existence proof, we  normalize the sequence of solutions $u_k$ so that it 
becomes locally uniformly bounded in $L^p$.
Let $B_0$ a ball centered at a point in 
$\partial\Om$, and denote by $B$  a ball centered at the same point so that $B_0\subset B$. 

\begin{lemma}\label{normalize} 
There exists a constant $C>0$ depending only on $p$, $\mu(B)/\mu(B_0)$, the radius of $B$, and the structure 
constants, such that for each $k\in \N$, with $w_k:=u_k-(u_k)_{B_0}$, one has
\[
\|w_k\|_{L^p(B)}\le C.
\]
\end{lemma}

\begin{proof} 
By the triangle inequality and Poincar\'e inequality one has
\begin{align*}
\vint_B|u_k-(u_k)_{B_0}|^p\, d\mu&\lesssim \vint_B|u_k-(u_k)_B|^p\,d\mu+|(u_k)_B-(u_k)_{B_0}|^p\\
  &\lesssim \vint_B|u_k-(u_k)_B|^p\,d\mu+\left(\vint_{B_0}|u_k-(u_k)_B|\, d\mu\right)^p\\
  &\lesssim \vint_B|u_k-(u_k)_B|^p\,d\mu+\frac{\mu(B)}{\mu(B_0)}\vint_B|u_k-(u_k)_B|^p\, d\mu\\
  &=\left[1+\frac{\mu(B)}{\mu(B_0)}\right]\vint_B|u_k-(u_k)_B|^p\, d\mu\\
  &\le \left[1+\frac{\mu(B)}{\mu(B_0)}\right]\, C\, \rad(B)^p\, \vint_{\lambda B}|\nabla u_k|^p\, d\mu.
\end{align*}
Consequently, 
\[
\int_B|u_k-(u_k)_{B_0}|^p\, d\mu\le C\, \rad(B)^p\, \left[1+\frac{\mu(B)}{\mu(B_0)}\right]\, \int_\Om|\nabla u_k|^p\, d\mu.
\]
Since the sequence $(\int_\Om|\nabla u_k|^p\, d\mu)_k$ is bounded, the proof  follows.
\end{proof}

Finally, we find a  function $u\in D^{1,p}(\Om)\cap L_{loc}^p(\Omega)$ that solves  the Neumann problem in every ball 
$B\subset \overline{\Omega}$ centered at a point on $\partial \Om$.

\begin{proposition}\label{lemma:convergence}
 In the hypotheses above, there exists $u\in D^{1,p}(\Om)$ such that 
\begin{equation}\label{eq-eq} 
\int_{\Omega} |\nabla u |^{p-2} \, \nabla u \cdot \nabla \psi d\mu = \int_{\partial \Omega \cap B} f T\psi d\nu,\end{equation}
for all $\psi \in D^{1,p}(\Omega)$.
\end{proposition}

\begin{proof}  
Since the sequence ${|\nabla u_k|^{p-2}\nabla u_k}$ is bounded in $L^{p'}(\Omega;\R^N)$, it converges weakly in $L^{p'}$ to a  
function $\Phi_1 \in L^{p'}(\Omega;\R^N)$. Letting $k\to \infty$ in \eqref{eq:Neum-prob-k} and invoking weak convergence 
on the left hand side it follows that for $\psi \in D^{1,p}(\Omega)$, one has
\begin{equation}\label{temp0}
\int_\Omega |\nabla u_k|^{p-2} \, \nabla u_k \cdot \nabla \psi \,d\mu \to \int_{\Omega} \Phi_1 \nabla \psi \,d\mu.
\end{equation}

By  Lemma \ref{panacea-bis} one can prove that 
\begin{equation}\label{temp0'}\int_{\partial \Omega } (f-f_k) T\psi d\nu 
\le C\bigg( \int_{\partial \Omega} |f(x)-f_k(x)|^{p'} J(x,y) \,d\nu(x)\bigg)^{\frac{1}{p'}} \| T\psi\|_{HB^\theta_{p,p}(\partial \Omega)},
\end{equation}
with $J$ as in \eqref{eq-panacea-bis} and for some  $y\in \partial \Omega$.
Combining Lemma \ref{L:weighted L^p convergence}, \eqref{temp0}, and \eqref{temp0'}, one obtains
\begin{equation}\label{eq-temp1}
\int_{\Omega} \Phi_1 \cdot \nabla \psi \,d\mu = \int_{\partial \Omega} f T\psi \,d\nu\end{equation}
for all $\psi \in D^{1,p}(\Omega)$.

On the other hand, from Lemma \ref{L:Cauchy} we note that the sequence
${\nabla u_k}$ converges strongly in $L^p(\Omega;\R^N)$ 
to a function $\Phi$, and hence (by passing to a subsequence) converges a.e., thus yielding 
$$\Phi_1= |\Phi|^{p-2}\Phi.$$
The latter and \eqref{eq-temp1} imply
\begin{equation}\label{eq-temp2}
\int_{\Omega} |\Phi|^{p-2} \, \Phi \cdot \nabla \psi \,d\mu = \int_{\partial \Omega } f T\psi \,d\nu,\end{equation}
for all $\psi \in D^{1,p}(\Omega)$.

For any ball $B \subset \overline{\Omega} $,  consider the sequence $w_k$  as in Lemma \ref{normalize}.
Invoking Mazur's lemma one can find convex combinations of the $w_k$ converging in $L^p(B)$ to a 
function $u_B\in L^p(B)$. 
Invoking \eqref{eq-temp2} and  applying \cite[Theorem 10]{FHK} we conclude that $\Phi=\nabla u_B$ almost everywhere in 
$B$. If $B' \subset \overline{\Omega} $ is another ball, then since $u_B,$ and $u_{B'}$ have the same gradient 
$\Phi$ on $B\cap B'$  they differ by a constant on $B\cap B'$. This allows us to define a function $u\in D^{1,p}(\Omega)$ 
such that for every $B \subset \overline{\Omega} $ there is a representative of the equivalence class of $u$ which agrees 
with $u_B$ on $B$ and for which \eqref{eq-eq} holds. 
To explicitly construct such function $u$ we proceed as follows: Consider an exhaustion of $\Omega$ by balls $B_j\subset B_{j+1}$. Arguing as 
in Lemma~\ref{normalize} we can apply Mazur's 
lemma to obtain
a convex combination sequence 
$w_{1,n}=\sum_{k=n}^{N_{1,n}}\lambda_{1,n,k}u_k$ and a function 
$v_1\in L^p(B_1)$ so that $w_{1,n}\to u_1$ both in $L^p(B_1)$ 
and pointwise a.e.~in $B_1$. Note that the sequence $(w_{1,n})_n$
is bounded in $L^p(B_2)$, and so via Mazur's lemma again we can
obtain a convex combination sequence 
$w_{2,n}=\sum_{k=n}^{N_{2,n}}\lambda_{2,n,k}w_{1,k}$ converging both
in $L^p(B_2)$ and pointwise a.e.~in $B_2$ to a function $v_2$.
Note that $B_1\subset B_2$, and so by the pointwise convergence
of $w_{1,k}$ to $v_1$ in $B_1$, necessarily $v_2=v_1$ a.e.~in
$B_1$. Moreover, $w_{2,n}$, being a convex combination of 
$w_{1,n}$ which is itself a convex combination of $w_k$, we know that
$w_{2,n}$ is a convex combination of $w_k$ as well. Proceeding in this fashion we can
inductively construct convex combination sequences $w_{m,n}$ of
$w_k$ such that $w_{m,n}$ converges in $L^p(B_m)$ and pointwise
a.e.~in $B_m$ to a function $v_m$ such that $v_m=v_l$ 
whenever $1\le l\le m$ on $B_l$.  Thus we can set $u$ on $X$ by
$u(x)=v_m(x)$ whenever $x\in B_m$.
\end{proof}

\section{Stability estimates for the Neumann problem}

The main results of this section are estimates that show continuity of the solution of the Neumann problem with respect to the 
weighted $L^{p'}$ norm of the data. In turn this will yield uniqueness of the solution (modulo a constant).

We begin with an a-priori estimate on the $D^{1,p}$ norm of a solution in terms of the data. By virtue of  
Theorem~\ref{thm:equiv-unbounded} and using an argument similar to the one in the proof of~\cite[Theorem 1.5]{CKKSS}, 
one can prove an analogue statement in our unbounded setting:

\begin{lemma}\label{estimate wrt f} 
For $f\in L^{p'}(\partial \Omega, \nu_J)$ denote by $u$ a corresponding solution of the Neumann problem with data $f$.
There exists $C>0$ depending only on $x_0$ and on the structure conditions such that 
\begin{equation}\label{E:data continuity}
\int_\Om |\nabla u(y)|^p d\mu(y) \le C \int_{\partial \Omega} |f(y)|^{p'} J(x_0,y) d\nu (y)\end{equation}
where  $J$ is as in \eqref{intro-weight}.
\end{lemma}

\begin{proof} Note that the minimum value of $I(v)$ must be non-positive in the variational formulation of the Neumann problem in 
Theorem~\ref{thm:equiv-unbounded} (this follows by choosing $v=0$ as a competitor). Hence one has
$$\int_{\overline{\Omega}}|\nabla u|^{p}\,d\mu\le p\int_{\partial{\Omega}}Tu f\,d\nu.$$
Applying Lemma \ref{panacea-bis} yields that for some $x_0\in \partial \Omega$ one has
\[
 \int_{\overline{\Omega}}|\nabla u|^{p}\,d\mu\le \bigg( \int_{\partial \Omega} |f(y)|^{p'} J(x_0,y) d\nu (y)\bigg)^{\frac{1}{p^\prime}} \| Tu \|_{HB^\theta_{p,p} (\partial \Omega) }.
 \]
The proof now follows from the trace theorem Theorem \ref{thm-trace-local-est}.
\end{proof}

An argument similar to the one in Lemma \ref{L:Cauchy} yields continuity with respect to the data in the weighted $L^{p'}$ norm.

\begin{proposition} 
There exists a constant $C>0$ depending only on the  PI constants and on $p$, such that 
for every  $f,g\in L^{p'}(\partial \Omega, \nu_J)$, if we denote by $u\in D^{1,p}(\Omega, d\mu)$ (resp. $v$)  the  
solution of the Neumann problem with data $f$ (resp. $g$) as in Lemma~\ref{lemma:convergence}, {then 
when $p\ge 2$ we have
\[
\int_{\Omega} |\nabla u - \nabla v |^p d\mu\le C\, \int_{\partial \Omega} |f - g|^{p'}(y) J(x,y)\, d\nu(y),
\]
and when $1<p<2$ we have
\begin{align*}
\int_{\Omega} |\nabla u - \nabla v |^p d\mu
\le C\, \| (|\nabla u|+|\nabla v|)\|_{L^p(\Omega)}^{p(2-p)}\, \Bigg( \int_{\partial \Om}|f(y)-g(y)|^{p'}J(x,y) d \nu(y)\Bigg)^{p-1},
\end{align*}}
where $J$ is as in \eqref{intro-weight}.
\end{proposition}

\begin{proof}
The proof follows the same outline as the one in Lemma \ref{L:Cauchy}. 
{As in Lemma~\ref{panacea-bis} we fix $x\in \partial\Om$.}
Since both $u$ and $v$ are solutions, they satisfy the identity (with the obvious modifications for $v$),
\[
\int_\Om|\nabla u|^{p-2}\, \nabla u \cdot \nabla \phi \, d\mu=\int_{\partial\Om} \phi\, f\, d\nu
\]
for any $\phi\in D^{1,p}(\Omega)$.

Choose $\phi=u-v\in D^{1,p}(\Omega)$. Subtracting one equation from the other yields
\begin{multline}
\int_{\Omega} (|\nabla u|^{p-2} \nabla u - \nabla v|^{p-2} \nabla v) \cdot (\nabla u- \nabla v) d\mu
\\
= \int_{\partial \Om} (f - g) (y) (u-v)(y) d\nu(y) 
\end{multline}
Applying Lemma~\ref{panacea-bis} and the trace theorem Theorem~\ref{thm-trace-local-est}, we obtain
\begin{multline}\label{E:bound-cauchy1}
\int_{\Omega} (|\nabla u|^{p-2} \nabla u - \nabla v|^{p-2} \nabla v) \cdot (\nabla u- \nabla v) d\mu 
 \\ \le \Bigg( \int_{\partial \Om}|f(y)-g(y)|^{p'}J(x,y) d \nu(y)\Bigg)^{\frac{1}{p'}}  \| \nabla u -\nabla v\|_{L^p(\Omega)}.
\end{multline}
Next we apply \eqref{eqn:monotonicity} and infer
 in the case $p\ge 2$,
\begin{equation}
\int_{\Omega} |\nabla u - \nabla v |^p d\mu  \le C \int_{\partial \Omega} |f - g|^{p'}(y) J(x,y) d\nu(y).
\end{equation}

If $1<p\le 2$ we apply H\"older inequality, \eqref{eqn:monotonicity} and \eqref{E:bound-cauchy1}, to show that
\begin{multline}
\|\nabla u - \nabla v\|_{L^p(\Omega)}^p \le \Bigg(\int_\Om(|\nabla u|+|\nabla v|)^p\Bigg)^{\frac{2-p}{2}} \!\!\!\Bigg(\int_\Om
|\nabla u - \nabla v|^2  (|\nabla u|+|\nabla v|)^{p-2} d\mu
\Bigg)^{\frac{p}{2}}
\\
\le \Bigg(\int_\Om(|\nabla u|+|\nabla v|)^p\Bigg)^{\frac{2-p}{2}}\Bigg(  \int_{\Omega} (|\nabla u|^{p-2} \nabla u - |\nabla v|^{p-2} \nabla v) \cdot (\nabla u- \nabla v) d\mu\Bigg)^{\frac{p}{2}} \\
\le \Bigg(\int_\Om(|\nabla u|+|\nabla v|)^p\Bigg)^{\frac{2-p}{2}} \Bigg( \int_{\partial \Om}|f(y)-g(y)|^{p'}J(x,y) d \nu(y)\Bigg)^{\frac{p}{2p'}}  \| \nabla u -\nabla v\|_{L^p(\Omega)}^{\frac{p}{2}}.
\end{multline}
Dividing both sides of the inequality by $ \| \nabla u -\nabla v\|^{\frac{p}{2}}_{L^p(\Omega)}$ yields the desired conclusion.
\end{proof}

Combining the latter and Lemma~\ref{estimate wrt f} we obtain the stability result.

\begin{theorem}\label{stability}
{Under} the hypotheses of the previous theorem, {when $p\ge 2$ we have
\[
\int_{\Omega} |\nabla u - \nabla v |^p d\mu\le C\, \int_{\partial \Omega} |f - g|^{p'}(y) J(x,y),
\]
and when $1<p<2$ we have
\begin{align*}
\int_{\Omega} |\nabla & u - \nabla v |^p d\mu\\ &\le C\,
\| (|f|+|g|)\|_{L^{p'}(\partial \Omega, Jd\nu)}^{p'(2-p)} \Bigg( \int_{\partial \Om}|f(y)-g(y)|^{p'}J(x,y) d \nu(y)\Bigg)^{p-1}.
\end{align*}}
\end{theorem}

\section{Construction of fractional $p$-Laplacians on $(Z, d_Z, \nu)$} \label{sec:construct-fractLap}

In this section, we introduce a definition for the fractional $p$-Laplacian in a doubling metric measure space and prove 
Theorem~\ref{thm:main-fract-Lap-intro}. Through the section
we fix $\theta\in (0,1)$, $p\in (1,\infty)$, and set $$\Theta=p(1-\theta).$$

Let $(Z,d_Z,\nu)$ be a complete, uniformly perfect, unbounded doubling metric measure space.
The  Besov semi-norm in the class $HB^\theta_{p,p}(Z)$ corresponds to the form
$\mathcal{E}_{p,\theta}$ given by
\[
\calE_{p,\theta}(u,v)
=\int_{Z}\int_{Z}\frac{|u(y)-u(x)|^{p-2}(u(y)-u(x))(v(y)-v(x))}{d(x,y)^{p\theta}\nu(B(y,d(x,y)))}\, d\nu(x)\, d\nu(y).
\]
Note that when $u,v\in HB^\theta_{p,p}(\partial\Omega)$, one has that $\calE_{p,\theta}(u,v)< \infty$ and that
$\calE_{p,\theta}(u,u)=\Vert u\Vert_{\theta,p}^p$. 

A plausible notion of fractional $p$-Laplacian would correspond to the nonlocal operator arising in the Euler-Lagrange equation 
for the Besov energy $\calE_{p,\theta}(u,u)$, and in the non-homogeneous case, would yield the nonlocal PDE
$$\calE_{p,\theta}(u,v)= \int_Z f v d\nu,$$
for all test functions $v\in HB^\theta_{p,p}(Z)$. 
However, in order to follow the approach of Caffarelli and Silvestre \cite{CS} and reap all its benefits, we will use a 
different, but equivalent, fractional energy functional, which is naturally associated to the Dirichlet-Neumann map for a 
uniformized version of an hyperbolic filling of $Z$.  The first step in this approach is a result of Butler 
 \cite{Bu1}, where  the construction of the hyperbolic filling in \cite{BBS} is extended to the non-compact case. A 
 synthesis of the work in \cite{Bu1} is  the following.
 
\begin{theorem}
     For each choice of $1<p<\infty$ and $0<\Theta<p$,
     there exists a uniform domain $\Omega$, subset of a metric measure 
     space $(\overline \Omega, d, \mu)$ that  satisfies hypotheses (H0), (H1), and (H2) in the introduction, and in 
     particular, $\mu$ is doubling and supports a $p$-Poincar\'e inequality.
     Moreover, the boundary $\partial \Omega$ agrees 
     with $Z$ as a set, and  when equipped with the metric induced by $(\Omega, d)$, it is bi-Lipschitz equivalent to $(Z,d_Z)$.  
\end{theorem}

As Butler writes in \cite{Bu1}, {\it ``... After a bi-Lipschitz change of coordinates on $Z$ we can then assume that 
$Z$ is isometrically identified with $\partial \Omega$''.}
 Applying the theorem above with such parameters $p,\Theta$   one obtains a uniform space $(\Omega, d,\mu)$ with 
 boundary $(\partial \Omega, d, \nu)$, bi-Lipschitz to $(X,d,\nu)$, where
the relation between $\nu$ and $\mu$ is expressed in terms of the co-dimension estimate \eqref{eq:Co-Dim-intro}. Since 
$(\Omega, d, \mu)$ satisfies a doubling condition and supports a Poincar\'e inequality, then we can choose a Cheeger differentiability structure as in 
Section~\ref{Sec:Cheegerishness} which we will fix for the rest of the section.

\bigskip

As in the previous sections, for $u\in HB^{\theta}_{p,p}(\partial \Omega)$ we set $\widehat{u}$ to be
the $p$-harmonic extension of $u$ in $(\Omega, d,\mu)$ 
as given by Theorem~\ref{thm:unbounded-Dirichlet}.
As mentioned in \eqref{eq:ET-go-home}, we follow the approach in \cite{CKKSS} and define a non-linear form that will 
yield the construction of the fractional $p$-Laplacian on $Z$.

\begin{definition}\label{def:E-sub-T}
We set $\mathcal{E}_T:HB^{\theta}_{p,p}(\partial\Om)\times HB^{\theta}_{p,p}(\partial\Om)\to\R$ by
\[
\mathcal{E}_T(u,v)=\int_{\Om}|\nabla \widehat u|^{p-2}\langle\nabla \widehat u,\nabla Ev\rangle\, d\mu,
\]
where  
$Ev$ is the extension of 
$v\in HB^{\theta}_{p,p}(\partial\Om)$ to $\Om$ as given in Theorem!\ref{thm-trace-local-est}.
\end{definition}

Note that whenever $w\in D^{1,p}(\Om)$ such that $T(w)=v$, we have that 
\[
\mathcal{E}_T(u,v)=\int_{\Om}|\nabla \widehat u|^{p-2}\langle\nabla \widehat u,\nabla w\rangle\, d\mu,
\]
because $\widehat{u}$ is $p$-harmonic in $\Om$ and $T(w-\widehat{v})=0$.

Following the proof in \cite[Lemma 6.1]{CKKSS}, and using the trace/extension results in Theorem ~\ref{thm-trace-local-est}, we have
the following lemma.

\begin{lemma}\label{lem:ETvsE}
There exists $C\ge 1$, depending only on the structure constants, such that for each $u\in HB^\theta_{p,p}(\partial\Omega)$, we have
\[
\frac{1}{C}\, \calE_T(u,u)\le \calE_{p,\theta}(u,u)\le C\, \calE_T(u,u).
\]
\end{lemma}

The latter implies that minimizers of the energy $ \calE_T(u,u)$ are global quasiminimizers of the energy $\calE_{p,\theta}(u,u)$, and 
vice versa.

\begin{proposition}\label{trace-neumann} 
 Let  $f\in L^{p'} (Z, \nu_J)\cap L^{p'}(Z, \nu)$ 
with $\int_Zf\, d\nu=0$. The minimizers of the fractional $p$-energy 
\[
 u\in HB^\theta_{p,p}(Z) \to I_f(u)=\frac{1}{p} \calE_T(u,u)-\int_Z u\, f \, d\nu
 \] 
 are exactly the traces of the solutions of the 
Neumann problem for the $p$-Laplacian on $\Omega$, and with Neumann data $f$. In particular, 
minimizers of $I_f$ are unique modulo a constant and satisfy the  nonlocal PDE 
\begin{equation}\label{weak-nonlocal} 
\calE_T(u,v)= \int_Z f v d\nu,
\end{equation}
for all test functions $v\in HB^\theta_{p,p}(Z)$.
\end{proposition}

\begin{proof} We start by proving that the traces of solutions of the Neumann problem for the 
$p$-Laplacian in $\Omega$ are minimizers of the functional $ I_f$. Let $u_f\in D^{1,p}(\Om)$ be a 
solution of the Neumann problem for the $p$-Laplacian on $\Om$ with data $f$. 
Recall that
$T(u_f)$ denotes its trace on $Z=\partial\Om$. Let  $v\in HB^\theta_{p,p}(Z)$; $\widehat v\in D^{1,p}(\bar \Om)$ 
is the unique solution to the Dirichlet problem for the $p$-Laplacian on $\Omega$ with boundary 
data $v$, and so $T(\widehat{v})=v$.  
Since $\widehat{T(u_f)}=u_f$, then invoking part (b) of Theorem \ref{thm:equiv-unbounded}) one has
\[
I_f(T(u_f))= \frac{1}{p} \int_\Om |\nabla u_f|^p d\mu - \int_Z T(u_f)f d\nu 
\le \frac{1}{p} \int_\Om |\nabla \widehat v|^p d\mu - \int_Z vf d\nu = I_f(v),
\]
proving that $T(u_f)$ is a minimizer of $I_f$.

Next we show that minimizers of the functional $I_f$ are traces of solutions of the Neumann 
problem for the $p$-Laplacian on $\Omega$, and as such are unique modulo a constant. Let 
$w\in HB^\theta_{p,p}(Z)$ be a minimizer of $I_f$ and denote by $\widehat w\in D^{1,p}(\Om)$ the 
corresponding solution of the Dirichlet problem for the $p$-Laplacian on $\Om$. Let 
$h \in D^{1,p}(\Om)$, 
and denote by $T(h)\in HB^\theta_{p,p}(Z)$ its trace to $Z$.
Since $\widehat h$ is a minimizer of the $p$-energy among all functions with the same trace, we have 
\begin{align}
\frac{1}{p} \int_\Om |\nabla \widehat w|^p d\mu - \int_Z w f d\nu = I_f(w) &\le I_f(T(h))\notag \\
&= \frac{1}{p} \int_\Om |\nabla \widehat {T(h)} |^p d\mu - \int_Z T(h) f d\nu\notag \\
&\le  \frac{1}{p} \int_\Om |\nabla h |^p d\mu - \int_Z T(h) f d\nu.
\end{align}
The latter, and part (b) of Theorem \ref{thm:equiv-unbounded}) yield that $\widehat w$ is a solution of the 
Neumann problem for the $p$-Laplacian on $\Om$ with data $f$ and with trace $w$.

To establish \eqref{weak-nonlocal} and conclude the proof, we observe that if $w\in HB^\theta_{p,p}(Z)$ is a 
minimizer of $I_f$, then $\widehat w$ solves the Neumann problem for the $p$-Laplacian on $\Omega$, and thus 
part~(a) of Theorem~\ref{thm:equiv-unbounded}) validates that $w$ satisfies~\eqref{weak-nonlocal}.\qedhere
\end{proof}

Proposition \ref{trace-neumann} motivates the following definition, which we  already mentioned in the introduction.
In what follows, the measure $\nu_J$ is defined in~\eqref{intro-weight-h}.

\begin{definition}\label{def: fractional p-laplacian} 
A function $u\in HB^\theta_{p,p}(Z)$ is in the domain of the fractional $p$-Laplacian operator
$(-\Delta_p)^{\theta}$ if there is a function $f\in L^{p'}(Z, \nu_J)$ such that the integral identity
\[
\mathcal{E}_T(u,\pip)=\int_Z\pip\, f\, d\nu
\]
holds for every  $\pip\in HB^\theta_{p,p}(Z)$. 
We then denote
\[
(-\Delta_p)^{\theta} u=f\in L^{p'}(Z,  \nu_J).
\]
\end{definition}

\begin{remark}
In view of Corollary \ref{normal-derivative} one has that
\[
    |\nabla u|^{p-2}\nabla u\cdot\nabla\eta_\epsilon\,d\mu\rightharpoonup -f\,d\nu.
\]
This shows that the fractional $p$-Laplacian in $(Z,d_Z)$ acts on the homogeneous Besov space 
$HB^\theta_{p,p}(Z)$ as the {\it Dirichlet-to-Neumann map} for the $p$-Laplacian in $(\Omega, d,\mu)$,
as in~\cite{GS3}.
\end{remark}

\begin{proposition}\label{thm:solving-fract-laplace}
Let $x_0\in Z$, and $J(x_0, \cdot)$,  $\nu_J$ be defined as in Theorem~\ref{intro-thm:Neumann}. 
For each $f\in L^{p'} (Z, \nu_J)\cap L^{p'}(Z, \nu)$ with $\int_Zf\, d\nu=0$
there is a function $u_f\in HB^\theta_{p,p}(Z)$ such that for each $\pip\in B^\theta_{p,p}(Z)$,
\[
\calE_T(u_f,\pip)=\int_Z\pip\, f\, d\nu.
\]
Moreover, there is a constant $C>0$, which depends solely on the structural constants of $\Om$ (or $Z$), and on the choice 
of the scalar product in the Cheeger differentiability structure such that
for each $f\in L^{p'}(Z)$,
\[
\calE_{p,\theta}(u_f,u_f)\le C\, \int_Z|f|^{p'}(y) J(x_0,y)\, d\nu(y).
\]
\end{proposition}

\begin{proof}
Invoking Proposition \ref{lemma:convergence} we denote by $u_f$ the unique solution to the Neumann problem for the 
Cheeger $p$-Laplacian in $\Omega$ with data $f$.
In view of the $p$-harmonicity of $u_f$, and the uniqueness for the Dirichlet problem, it follows that $\hat{u_f}=u_f$ and 
\[
\calE_T(u_f,v)=\int_\Om |\nabla u_f|^{p-2}\nabla u_f\cdot\nabla E v\, d\mu=\int_Z v\, f\, d\nu,
\]
for any $v\in HB^\theta_{p,p}(Z)$.
Moreover, by Lemma~\ref{lem:ETvsE},
\[
\calE_{p,\theta}(u_f,u_f)\le C\, \calE_T(u_f,u_f)=C\, \int_\Om|\nabla u_f|^p\, d\mu.
\]
Combining the above with Lemma~\ref{estimate wrt f} yields the desired inequality
$\calE_{p,\theta}(u_f,u_f)\le C\, \int_Z|f(y)|^{p'} \ J(x_0,y)\, d\nu(y)$.

\end{proof}

Proposition \ref{thm:solving-fract-laplace} proves part (1) of Theorem \ref{intro-thm:Neumann}.
In the same spirit, and following the approach in \cite{CKKSS}, the remaining properties  for the solutions of the nonlocal equation
$$(-\Delta)^\theta u=f$$
in $Z$ can be inferred from the analogue results for solutions of the Neumann problem for the $p$-Laplacian in $(\Omega, d, \mu)$. 
Hence part (2) in Theorem \ref{intro-thm:Neumann} follows immediately from Theorem \ref{stability}, while the 
boundary H\"older regularity, and the Harnack inequality follow from the analogue results for the $p$-Laplacian 
proved in \cite{CKKSS} (see Remark~\ref{stability})

\begin{remark}
It is natural here to ask why the measure $\nu_J$ as in~\eqref{intro-weight-h} is needed at all. We argue here that 
this measure $\nu_J$ arises naturally. As pointed out in~\cite{GS3}, we know that given a $p$-harmonic extension $\widehat{u}$
to $\Om$ of a function $u\in HB^\theta_{p,p}(Z)$ with $Z=\partial\Om$, the linear map $T_u:HB^\theta_{p,p}(Z)\to\R$ given by
\[
T_u(v):=\int_\Om|\nabla \widehat{u}|^{p-2}\, \nabla\widehat{u}\cdot\nabla\widehat{v}\, d\mu
\]
is a bounded map in the energy seminorm of $HB^\theta_{p,p}(Z)$. On the other hand, if $u$ arose as the solution to the 
problem posed in Definition~\ref{def: fractional p-laplacian}, then we also have that
$T_u(v)=\int_Z f\, v\, d\nu$. Thus we must have that
\[
|T_u(v)|\le C_f\, \mathcal{E}_{p,\theta}(v)^{1/p}.
\]
Note that $\int_Zf\, d\nu=0$. 
We now try to find a bound on $C_f$ as follows. 
We fix a point $x_0\in Z$ and consider $x\in B(x_0,R)$ for some fixed $R>0$.
When $v\in HB^\theta_{p,p}(Z)$,  we have
\begin{align*}
\int_Z&f(y)v(y)\, d\nu(y)=\int_Zf(y)\, [v(y)-v(x)]\, d\nu(y)\\
  &=\int_Z\, \left[f(y)\, \nu(B(x,d(x,y)))^{1/p}\, d(x,y)^\theta\right] \left[\frac{v(y)-v(x)}{\left[\nu(B(x,d(x,y)))^{1/p}\, d(x,y)^\theta\right]}\right]
     d\nu(y).
\end{align*}
It follows from H\"older's inequality that 
\begin{align*}
\nu&(B(x_0,R)) \Bigg\vert\int_Zf(y)v(y)\, d\nu(y)\Bigg\vert\\
\le&\,\, \mathcal{E}_{p,\theta}(v)^{\frac{1}{p}}\, \left(\int\limits_{B(x_0,R)}\int\limits_Z|f(y)|^{p'}\, \nu(B(x,d(x,y)))^{\frac{p'}{p}}\, d(x,y)^{\theta p'} d\nu(y)\, d\nu(x)\right)^{\frac{1}{p'}},
\end{align*}
that is, 
\[
|T_u(v)|\le \frac{\mathcal{E}_{p,\theta}(v)^{\frac{1}{p}}}{\nu(B(x_0,R))}\left(\int\limits_{B(x_0,R)}\int\limits_Z|f(y)|^{p'}\, J(y,x) \,d\nu(y)\, d\nu(x)\right)^{\frac{1}{p'}},
\]
where $J$ is as in~\eqref{intro-weight}.
It follows that
\[
C_f\le \frac{1}{\nu(B(x_0,R))}\left(\int\limits_{B(x_0,R)}\int\limits_Z|f(y)|^{p'}\, J(y,x) \,d\nu(y)\, d\nu(x)\right)^{\frac{1}{p'}}.
\]
\end{remark}

\section{Concluding remarks, motivations, and examples}\label{S:examples}

Roughly speaking, in this paper we approach the study of fractional powers of $p$-Laplacian operators  from the point of view that the non-local energies in a metric measure space 
$(Z,d,\nu)$ arise as  the \emph{trace of a local energy in a larger metric measure space}, as in problems
related to understanding behaviors of electric potentials on the surface of a physical or biological object, when such a surface behavior 
is determined by the physical and chemical processes going on within the interior of the object itself. 
This point of view can be directly adopted whenever $Z$ is the boundary of a uniform space $\Omega$ supporting a Poincar\'e inequality
(as in the case of \cite{CS} for instance). However, in some
situations, one only has the metric measure space $(Z,d,\nu)$ as data, without $Z$ being identified as the boundary of a uniform domain.
When this is the case, depending on how rich the metric measure structure on $Z$ is, there are several ways to construct 
a uniform domain $\Omega$ supporting a Poincar\'e inequality and that has $Z$ as a boundary. In the most general 
setting of the present paper, when $Z$ is only doubling, we obtain the domain $\Om$ as a uniformization of one of its 
hyperbolic fillings, see for instance~\cite{BBS}.

In this case, $Z$ has a natural identification with the boundary of $\Om$ via a bi-Lipschitz map. In the case that $(Z,d,\nu)$ also
satisfies a $p$-Poincar\'e inequality, we can consider $\Om:=Z\times(0,\infty)$, with $\partial\Om=Z\times\{0\}$ isometric to $Z$,
and this approach was taken in~\cite{EbGKSS} for the case $p=2$ and in~\cite[Appendix]{CKKSS} for $1<p<\infty$.

In this section we compare our approach and our results with the existing literature on fractional $p$-Laplacian type operators. 
This extant body of work   has been mostly developed in the Euclidean setting, although some authors (see for 
instance \cite{AMRT,FF, Gar,PaPi} and 
references therein) have tackled the sub-Riemannian  setting. We show in Remark 9.5 and Example 9.9 how our approach is related to the class 
of variable coefficient nonlocal $p$-Laplacian operators as described in \cite{Caff} and developed in  \cite{BBK,CKP,GZZ,KKL}. We 
also argue how potentially it could greatly simplify the proofs of the existing results (we use as an example the nonlocal Harnack inequality), 
furthermore providing a unifying framework for a broad variety of geometric settings. 

\medskip

To do so, we will proceed through a series of settings, where the structure of the underlying space $(Z,d,\nu)$ becomes richer and richer, 
from doubling metric measure spaces, to doubling spaces supporting a  Poincar\'e inequality, all the way 
to the sub-Riemannian Heisenberg group and the Euclidean spaces. As more structure is added we show how the results can be refined, and how they relate to the existing literature.

\begin{subsection}{Literature review and comparisons when $Z$ is doubling and supports a Poincar\'e inequality}\label{PI section}
In the special case when $(Z, d, \nu)$ itself supports a $p$-Poincar\'e inequality, 
then following the approach in \cite{EbGKSS}, we 
{ view $Z$ as the boundary of the domain $\Om:=Z\times(0,\infty)$, with $\overline{\Om}=X=Z\times[0,\infty)$,}
endowed with the product distance $d_X$ and the measure $d\mu_a= |y|^a \,dy\, d\nu$ where $a=1- p\theta$.
This space $(X,d_X,\mu_a)$ is also a doubling metric measure space supporting a $p$-Poincar\'e inequality, 
as established in~\cite{EbGKSS}, and it is also a uniform 
space (see~\cite[Proposition 4.1]{EbGKSS}). This space satisfies (without need of any further uniformization) the 
hypotheses (H0), (H1) and (H2), and in addition one has that $\partial X$, with the metric induced by $d_X$, is indeed 
isometric (and not just bi-Lipschitz equivalent) to the original metric space $(Z,d_Z)$.

To give a concrete example, in the case 
in which $Z=\R^n$, $d$ is the Euclidean distance and $\nu$ is the Lebegue measure, then we obtain an extension of the 
fractional $p$-Laplacian on $\R^n$ through the Neumann problem  in $(x,y)\in \R^n \times \R^+$,
\begin{equation}\label{eqn91}
\begin{cases}
\text{div}\bigg( |\nabla u(x,y)|^{p-2} y^a \nabla u(x,y)\bigg) =0 \text{ for } y>0 \text{ and } x\in \R^n\\
\lim_{y\to 0} |\nabla u(x,y)|^{p-2} y^a \partial_y u (x,y) = f(x) \text{ at }x\in \R^n,
\end{cases}
\end{equation}
where $\text{div}$ and 
$\nabla$ refers to the usual differential structure in the Euclidean domain $\R^n\times(0,\infty)$
endowed with Lebesgue measure. The equation \eqref{eqn91} is reminiscent of the point of view in \cite{SiVa}.
If $u\in D^{1,p}(\R^n\times \R)$ is a solution of this Neumann problem, then its trace $Tu$  
on 
the boundary $\R^n=\partial(\R^n\times(0,\infty))$ 
satisfies the fractional partial differential equation
\[
(-\Delta_p)^{\theta} Tu=f.
\]
Here $\theta=\tfrac{1-a}{2}$ with $-1<a<1$.
Our results show that both problems have  unique (modulo additive constants) solutions when $f\in L^{p'}(Z,  \nu_J)$, in 
addition to regularity and stability properties.

Note that the energy form we use on the Besov class $B^\theta_{p,p}(Z)$ is $\mathcal{E}_T$, which in general is 
equivalent (see Lemma \ref{lem:ETvsE}) but  distinct from
the energy form $\mathcal{E}_{p,\theta}$.
Although we do not know whether an equality between $\mathcal{E}_T$ and $\mathcal{E}_{p,\theta}$  may hold with special 
choices of Cheeger differential structures on $Z$,  
we do know that   given a choice of Cheeger differential structure on $Z$ itself, 
the results in \cite{EbGKSS} and \cite[Section 7]{CKKSS}  show that when $p=2$ the fractional 
operator we construct through the approach in this paper coincides with $(-\Delta_C)^\theta$ defined through the 
classical spectral decomposition, where $\Delta_C$ is the Cheeger-Laplacian on $Z$ obtained using the theory of
Dirichlet forms~\cite{FOT}. When $Z=\R^n$, $p=2$,
and the Cheeger differential structure is the classical Euclidean differential structure on $Z$, then the two energy forms
coincide, as shown in~\cite[Section 3.2]{CS}. Likewise, when $p=2$ our results also include those proved 
in~\cite{FF} in the context of sub-Riemannian Carnot groups.
\end{subsection}

\begin{subsection}{Literature review and comparisons when $Z$ is the Euclidean space or the Heisenberg group}\label{euclidean}
In the case $p\neq 2$, the fractional powers of the $p$-Laplacian in the Euclidean, Riemannian and sub-Riemannian 
Heisenberg group setting have been 
studied in depth in recent years (see for instance \cite{BBK,BLS, CKP, dTGCV, GZZ,IMS, KKL, PaPi}
and the references therein). 
	 The version of fractional $p$-harmonic functions as considered in those
papers is different from ours, as these works consider  functions that minimize $HB^\theta_{p,p}(Z)\ni u\mapsto\mathcal{E}_{\theta,p}(u,u)-p\int_Z u\, f\, d\mu$,
with many focusing on the homogeneous case $f=0$.
As far as we are aware, with the exception of \cite{SiVa}, there does not seem 
to be published literature on an analogue of the Caffarelli-Silvestre approach in the nonlinear setting $p\ne 2$,  
even in the Euclidean setting. In 
our work, by substituting the fractional energy $\mathcal{E}_{p,\theta}$ with the equivalent one $\mathcal{E}_{T}$, and 
studying its critical points, we are able to carry through the whole Caffarelli-Silvestre approach in the more general setting 
of doubling metric measure spaces. To further compare the two approaches, we note that both the Euclidean 
space and the sub-Riemannian Heisenberg groups are spaces 
that support a Poincar\'e inequality
when endowed with the Lebesgue measure.  
So, as described in Subsection~\ref{PI section}, we 
can then apply our construction with $Z$ being a bounded open $U$  
set in either of these  spaces. We can take as uniform domain  $\Om$ the uniformization of the product 
$\Om=Z\times (0, \infty)$, as in \cite[Section 7]{CKKSS} and \cite{EbGKSS}.  Another viable alternative is to 
choose as $Z$ the whole Euclidean space $\R^n$ or a sub-Riemannian Heisenberg group. Then the 
product $Z\times (0, \infty)$ will satisfy all the hypotheses (H0),(H1),and (H2) in the present paper without 
need of a uniformization. 
Note that the weak solutions to the Euclidean problem (as in, for instance \cite{CKP}) or its Heisenberg 
group counterpart (as in \cite{PaPi}) are minimizers of the functional
  $\mathcal{E}_{p,\theta} (u,u)$.  Invoking Lemma \ref{lem:ETvsE}, and the energy minimization of 
  $p$-harmonic functions,  we have that such functions are global quasiminimizers of the functional 
 \[
 \mathcal{E}_T (u,u)= \int_\Om |\nabla \widehat{u}|^p d\mu,
 \]
 i.e. there exists a constant $C>0$ such that for all  $v\in N^{1,p}(\Om)$ one has
 $$
\mathcal{E}_T (u,u)\le C\int_\Om |\nabla (\widehat{u}+v)|^p d\mu.
 $$
 As mentioned above, in the case $p=2$ the two energies  $\mathcal{E}_{p,\theta} (u,u)$ and $ \mathcal{E}_T (u,u)$ coincide and consequently $u$ is a global minimizer of $\mathcal{E}_T (u,u)$ (and thus also a local minimizer). For general $p\neq 2$, at this time {\it we do not know whether $\hat u$ is also a local quasiminimizer}, i.e. if there exists $C>0$ such that for all open subsets  $U\subset  \bar U\subset \Omega$ and $v$ with zero trace in $U$, such that one has
 $$\int_U |\nabla \widehat{u}|^p d\mu \le C  \int_U |\nabla (\widehat{u}+v)|^p d\mu.$$

 Next, we show how {\it    if the latter holds} then we can apply  our method to prove  the Harnack inequality in \cite{CKP}  for weak   solutions of 
 \begin{equation}\label{CKP}
\begin{cases} (-\Delta_p)^\theta u=0  &\text{ in } U \\ u=g & \text{ in } \R^n\setminus U
\end{cases}
\end{equation} 
where $U$ is a bounded open subset of $\R^n$, $u>0$ in $\R^n$ (i.e. $u$ has vanishing tails)
 and $(-\Delta_p)^\theta$ is defined through the Besov energy norm $\mathcal{E}_{p,\theta}$.
  From the local quasiminimization property one could  conclude from~\cite{KiSh} that $\widehat{u}$ is locally H\"older continuous and therefore
so is $u$. In view of the fact that $u$ has finite energy, then we could invoke the minimum principle~\cite[Corollary 6.4]{KiSh} 
  which implies that $\hat u>0$ in $X$.  
We also recall that thanks to \cite[Proposition 7.1]{AS} one has that the space $\Om \cup B(x_0,R)$ is also a doubling space that supports a Poincar\'e inequality.  
Invoking~\cite[Corollary~7.3]{KiSh}, one has that $\widehat{u}$
satisfies a Harnack inequality on all balls $B$ such that $4B\subset \Om\cup B(x_0,R)$.  In particular, for each such ball one has
\[
\sup_{B\cap B(x_0,R)} u \le \sup_{B} \widehat{u} \le C \inf_B \widehat{u} \le C\inf_{B\cap B(x_0,R)} u,
\]
yielding the Harnack inequality. The same argument would yield the Harnack inequality  in \cite{PaPi} in the case where there are no tails.
\end{subsection}

\begin{subsection}{Literature review and comparisons for non-homogeneous problems}
Both the papers~\cite{CKP14} and~\cite{PaPi} deal with the homogeneous equation $(-\Delta_p)^\theta u=0$. In  their context, it is also possible to
make sense of the inhomogeneous equation
$(\Delta_p)^\theta u=f$. In the 
Euclidean setting, the inhomogeneous problem for the fractional $p$-Laplacian and its parabolic counterpart were studied in the 
paper~\cite{DZZ}. They show that when
the inhomogeneity data $f$ is bounded and $p>2$, then the solution is H\"older continuous, see~\cite[Theorem~3]{DZZ}.
In~\cite{IMS} a similar study was undertaken for the (time-independent) equation in $C^{1,1}$-domains in Euclidean spaces, 
with the solution required to vanish in the complement of the domain. Here again the inhomogeneity data was assumed to be 
bounded, and global H\"older continuity estimates were established for the solutions, see~\cite[Theorem~1.1]{IMS}.
For $p=2$ a  more general situation of Lipschitz domains in Euclidean spaces, with inhomogeneity function that also belongs to the 
dual of the Besov space, was considered in~\cite[Theorem~1.2]{CSt}. The limitations on the inhomogeneity data placed
in~\cite[Theorem~1.2]{CSt} agrees with those given in Theorem~\ref{intro-thm:Neumann-regularity} of the present
paper,
for there $p=2$, $1-\Theta/p=s$,
$\Theta=1$ and $Q_\mu=n+1$ where $n$ is the dimension of the Euclidean space in which the domain is located.
The papers~\cite{CSt,DZZ,IMS} involve operators defined through the form $\mathcal{E}_{p,\theta}$.

In the
metric setting considered in the present paper, we are able to consider a  more general class of  inhomogeneous data. In fact, we  prove
Harnack inequalities and local H\"older continuity of solutions for higher integrability classes of inhomogeneous  
data $f$, which are not necessarily bounded. We emphasize that however our results deal with   the
operator that corresponds to $\mathcal{E}_T$ rather than the one corresponding to $\mathcal{E}_{p,\theta}$. As 
a consequence 
we cannot directly apply our theorems to infer results 
about the solutions in \cite{DZZ, IMS}. 
\end{subsection}

\begin{subsection}{On conditions related to H\"older regularity of solutions} \label{sub:Holder}
Next, we want to address the optimality of our integrability conditions on the data $f$ in the proof of the global  
H\"older regularity for solutions of the Neumann problem in Theorem \ref{intro-thm:Neumann-regularity}.
Suppose we have that $f\in L^q(B(x_0,10r_0))$ with $f\ge 0$ on this ball. 
For which ranges of $q$ can one expect to obtain H\"older regularity? We recall 
a result of M\"ak\"al\"ainen 
\cite[Theorem~1.3]{Mak}: Let $D$ be open and bounded subset of a measure metric space satisfying condition (H1) in the introduction. 
Consider a weak solution $u\in N^{1,p}_0(D)$ of the PDE
\begin{equation}\label{makalainen-pde}
\int_D |\nabla u|^{p-2} \nabla u\cdot \nabla \phi \,d\mu= \int_D \phi \,d\widehat{\nu},
\end{equation}
for any $\phi\in N^{1,p}_0(D)$ and where $\widehat{\nu}\in (N^{1,p}_0(D))^*$ is a Radon measure.
If $\alpha>0$ is sufficiently small, depending on the doubling and Poincar\'e constants in (H1), 
then the solution $u$ is $\alpha$-H\"older continuous locally in $D$ 
if and only if there is a constant $M>0$ such that 
\begin{equation}\label{makalainen-condition}
\frac{\widehat{\nu}(B(x,r))}{\mu(B(x,r))} \le M r^{-p+\alpha (p-1)},
\end{equation}
for all balls $B(x,4r) \subset D$. 

In our setting, we are allowed to apply this result to the case that $D=B(x_0,10r_0)$ for some $x_0\in\partial\Omega$ 
and $r_0>0$ such that $f\geq 0$ on $B(x_0,10r_0)$, 
and $\widehat{\nu}$ the weighted measure $\nu_f$ given by $d\nu_f=f\, d\nu$. Observe that $B(x_0,10r_0)$ is open in $\overline{\Om}$. 
We want to see for which range of $q$  the Radon measure  $\nu_f$ supported on $\partial\Om$  
satisfies 
condition~\eqref{makalainen-condition}. 
Consider $x\in B(x_0,r_0)$ and $0<r<r_0$. Then, given that $f\in L^q(\partial\Om)$,
an application of
H\"older's inequality yields
\[
\nu_f(B(x,r))=\int_{B(x,r)}f\, d\nu\le \left(\int_{B(x_0,10r_0)}f^q\, d\nu\right)^{1/q}\, \nu(B(x,r))^{1/q'}.
\]
Consequently, invoking the co-dimension hypothesis \eqref{eq:Co-Dim-intro} and recalling the lower mass dimension 
from~\eqref{eq:lower-mass-exp},
\begin{align}\label{eq:tight1}
\frac{\nu_f(B(x,r))}{\mu(B(x,r))}&\le \left(\int_{B(x_0,10r_0)}f^q\, d\nu\right)^{1/q}\, \frac{\nu(B(x,r))^{1/q'}}{\mu(B(x,r))}\notag\\
 &\approx  \left(\int_{B(x_0,10r_0)}f^q\, d\nu\right)^{1/q}\, \frac{r^{-\Theta/q'}}{\mu(B(x,r))^{1/q}}\notag\\
 &\lesssim  \left[\left(\int_{B(x_0,10r_0)}f^q\, d\nu\right)^{1/q}\, \frac{r_0^{Q_\mu/q}}{\mu(B(x_0,r_0))^{1/q}}\right] r^{(-\Theta/q'-Q_\mu/q)}.
\end{align}
In order to apply \eqref{makalainen-condition}, we need to identify 
a H\"older continuity exponent $\alpha\in (0,1)$ for the solution $u$ such that 
\begin{equation}\label{ranges}
r^{-p+\alpha(p-1)}\ge r^{(-\Theta/q'-Q_\mu/q)}=r^{-(\Theta(q-1)+Q_\mu)/q}.
\end{equation}
Note that  $p>\alpha(p-1)$  as $0<\alpha\le 1$. Since we only focus on $0<r\le r_0$, without loss of generality we assume that   
$r<1$; thus, for \eqref{ranges} to hold we must have that $\alpha$ needs to satisfy
\[
p-\alpha(p-1)\ge \frac{\Theta(q-1)+Q_\mu}{q}.
\]
It follows that we must have
\[
\alpha\le \frac{pq-Q_\mu-\Theta(q-1)}{q(p-1)}.
\]
As we also want $\alpha>0$, the above estimates indicate that we need $pq-Q_\mu-\Theta(q-1)>0$, that is,
\[
q>\frac{Q_\mu-\Theta}{p-\Theta}.
\]
This is  exactly the condition we have for $q$ in Theorem~\ref{intro-thm:Neumann-regularity}(1). This argument also shows that if $f$ 
supports the growth estimate~\eqref{eq:tight1} and $u$ is H\"older continuous then, independently of the H\"older exponent, 
one must have  $q>\frac{Q_\mu-\Theta}{p-\Theta}$.

\medskip

Next we turn to the H\"older exponent. Note that through the argument above, in the regions of $\partial\Om$
where $f\ge 0$, we obtain the H\"older exponent 
\[
\eps\le \left(1-\frac{q_0}{q}\right)\, \frac{p-\Theta}{p-1},
\]
whereas in Theorem~\ref{intro-thm:Neumann-regularity}(1) we prove the H\"older regularity for a \emph{smaller} upper bound for $\eps$,
\[
\eps\le \left(1-\frac{q_0}{q}\right)\, \frac{p-\Theta}{p},
\]
 with the \emph{same} constraint on $q$ (but without the extra positivity hypothesis on $f$). Recently we were able to extend the results in  Ono~\cite{ono}, 
 who proved the analogue result for measures that can change sign in the Euclidean case, to our general setting. Hence we can improve the proof of the regularity in~\cite{CKKSS} and 
 obtain the optimal H\"older exponent also in the case when $f$ switches sign. These results will appear in a separate paper.

Euclidean predecessors of the results in M\"ak\"al\"ainen~\cite{Mak} can be found for instance in~\cite{CSt,KilZ, RakZi}.
\end{subsection}

\begin{subsection}{Literature review and comparisons for variable coefficients fractional operators}
One of the driving motivations for the study of analysis in metric space comes from the fact that  the study of weak 
solutions of  local quasilinear elliptic PDE in divergence form $\text{div}(A(x,\nabla u))=0$ modeled on the $p$-Laplacian 
(the so-called $\mathcal A-$harmonic functions in \cite{HKM}) in the Euclidean space can be reframed in the context of 
analysis in measure metric spaces by suitably choosing a new metric and a new measure so that the minimizers of the 
corresponding Dirichlet energy coincide with the solutions of these PDEs. 

In the following and in Example 9.9, we show how a  similar reframing can be done, to some extent, in the nonlocal setting through the ideas 
in this paper and in \cite{CKKSS}.

In  an effort to capture the variable coefficients versions of fractional $p$-Laplacians, in the recent Euclidean literature on 
nonlocal PDEs (see for instance \cite{BBK,CK1, CK2, CKK, CKW,CKP, KKL} and references therein)  several authors consider operators 
arising as Euler-Lagrange equations of functionals of the form
\begin{equation}\label{k-energy}
 \int_{\R^n} \int_{\R^n} K(x,y) |u(x)-u(y)|^p dx dy
 \end{equation}
where $K$ is symmetric 
and is comparable to  $|x-y|^{-(n+p\theta)},$ i.e. there is a constant $C>0$ such that
\[
C^{-1} |x-y|^{-(n+p\theta)} \le K(x,y) \le C |x-y|^{-(n+p\theta)}.
\] 
We refer the interested reader on this topic to the recent paper~\cite{KuNoSi} for further developments of the regularity theory  when $p=2$. 
In order to include this class of nonlocal operators in our framework we will need to choose $Z=\R^n$ and 
suitably modify the distance function, while keeping the Lebesgue measure $d\nu=dx$. We introduce a new 
metric $d_K$ in $\R^n$ by first defining the quasi-distance $\hat d_K (x,y) = K(x,y)^{1/(-n-p\theta )}$,
and then 
setting $d_K$ to be the corresponding  length distance. It is a simple exercise to show that $d_K$ is bi-Lipschitz 
to the Euclidean norm, and as such the metric measure space $(Z, d_K, \nu)$ is a space
supporting a Poincar\'e inequality. We also have that 
\begin{equation}\label{K-i}  \lambda^{-1} K(x,y) \le d_K (x,y)^{-(n+p\theta)} \le \lambda K(x,y)\end{equation}
for some constant $\lambda>0$. Hence the Besov spaces corresponding to $(Z, d_K, \nu)$ 
coincide with those for the Euclidean space, and minimizers of one Besov energy are quasiminimizers of the 
other. In particular, the minimizers of \eqref{k-energy} are global quasiminimizers of the $\mathcal{E}_T (u,u)$ energy.

Note that in case $K$ is already a distance raised to the power
$-(n+p\theta)$, or it is a product of a  distance raised to the power $-(n+p\theta)$ times an $A_p$ Muckenhoupt 
weight $\omega$, then the inequalities
in \eqref{K-i} become identities, once we change the measure $d\nu=\omega dx$. 
\end{subsection}

\begin{subsection}{Modeling through choices of Cheeger structures}\label{modeling}
One of the main features in our approach in defining fractional powers of $p$-Laplacians is the possibility of choosing  Cheeger differential structures. This further degree of freedom yields more flexibility than it may appear in modeling rough, possibly anisotropic
geometries.   The Cheeger differential
structure is based on first-order Taylor approximation theory rather than the theory of distributions and integration by parts.
However, what is important to us is that the energy induced by this structure is comparable to the energy obtained
from the minimal weak upper gradients theory of Newton-Sobolev spaces, together with the fact that the Cheeger gradient
energy minimization lends itself to an Euler-Lagrange equation. 
The next example gives a toy model demonstration of this paradigm. We consider the Euclidean space $\R^2$, with its Euclidean metric and Lebesgue measure, but choose an anisotropic Cheeger structure for which the corresponding $p$-Laplacian cannot be modeled through minimizing the  $\mathcal{E}_{p,\theta}$ energy even in the case $p=2$. 
\end{subsection}

\begin{example} We show through a simple example how one could capitalize on our approach (which includes the freedom to choose Cheeger structures) in order to  create  models for nonlinear diffusion in an anisotropic material, where in
some regions of the space the diffusion speed is faster in one direction.  
As a toy model, in a highly idealized situation, one could consider $Z=\R^2$, equipped with the Euclidean metric, Euclidean inner product and $2$-dimensional Lebesgue measure $\mathcal{L}^2$. For our purposes  the usual (isotropic) Euclidean differential
structure would  not work, although it is a perfectly fine choice of a Cheeger differential structure. We instead choose the following Cheeger structure: For Lipschitz functions $u$ on $X$, we set
\[
\nabla_C u(x,y):=\begin{cases} (\partial_x u(x,y), \partial_y u(x,y))&\text{ if }x>0,\\
   (2\, \partial_x u(x,y), \partial_y u(x,y))&\text{ if }x\le 0.\end{cases}
\]

For simplicity, let's first start with the linear diffusion case $p=2$. We choose the Euclidean inner product to pair with this structure as well, which in turn yields a Dirichlet form in the
sense of~\cite{FOT}, from whence we obtain the corresponding Laplacian-type operator $\Delta_C$, called 
the infinitesimal generator in~\cite{FOT}. The domain of the operator $\Delta_C$ is the collection of all
functions $u\in W^{1,2}(\R^2)=N^{1,2}(\R^2)$ for which there is a \emph{function} $f_u$ so that 
\[
\mathcal{E}(u,v):=\int_{\R^2}\nabla_Cu\,\cdot\,\nabla_Cv\, d\mathcal{L}^2=-\int_{\R^2}v\, f_u\, d\mathcal{L}^2
\]
whenever $v\in W^{1,2}(\R^2)$; in other words, $\Delta_Cu$ is defined through an operator in divergence form.
The domain of the Laplacian $\Delta_C$ associated with this differential structure is not the domain of the Euclidean Laplacian $\Delta$.
Indeed, if $u\in C^\infty(\R^2)$ is in the domain of $\Delta_C$, then for every compactly supported function $\pip\in C^\infty_c(\R^2)$
we must have that 
\begin{align*}
-\int_Z\Delta_Cu\, d\mathcal{L}^2&=\int_Z\, \nabla_Cu\cdot\nabla_C\pip\, d\mathcal{L}^2\\
  &=\int_{\R^2}\nabla u\cdot\nabla \pip\, d\mathcal{L}^2+3\int_{L\times\R}\partial_xu\, \partial_x\pip\, d\mathcal{L}^2\\
  &=-\int_{\R^2}\pip\, \Delta u\, d\mathcal{L}^2-3\int_{L\times\R}\pip\, \partial_x^2u\, d\mathcal{L}^2
  +3\int_Y\pip(0,y)\, \partial_xu(0,y)\, dy,
\end{align*}
where $L=\{x\in\R\, :\, x<0\}$ and $Y:=\{(0,y)\, :\, y\in\R\}$ is the $y$-axis in $Z$. Thus $\Delta_Cu$ is given by the Radon measure
\begin{align*}
\Delta_Cu(x,y)&\,d\mathcal{L}^2(x,y)\\
 =&-3\partial_xu(0,y)\, d\mathcal{H}^1\vert_Y(y)
+\left[3\chi_L(x)\, \partial_x^2\, u(x,y)+\Delta u(x,y)\right]\, d\mathcal{L}^2(x,y)
\end{align*}
where
$\Delta u$ is the Euclidean Laplacian of $u$.
As $u$ is in the domain of $\Delta_C$, it must follow that $\Delta_Cu$ is absolutely continuous with respect to $\mathcal{L}^2$, and
so we must have that $\partial_xu(0,y)=0$. Meanwhile, the domain of the Euclidean Laplacian $\Delta$ does not have this restriction.
Thus there are functions in the domain of $\Delta$ that are not in the domain of $\Delta_C$. 

On the other hand, the function 
$u(x,y)=\varphi(y)\,v(x)$ with $\varphi$ a compactly supported smooth function on $\R$ with $\varphi(0)=1$ and 
$v(x)=(1+3\chi_{\R\setminus L}(x))\, x$ belongs to the domain of $\Delta_C$ but not to the domain of $\Delta$.
Indeed, with $A$ the $2\times 2$ matrix given by
 \[
   A(x)=
  \left[ {\begin{array}{cc}
   1+\chi_L(x) & 0 \\
   0 & 1 \\
  \end{array} } \right],
\]
we have that \begin{equation}\label{divform} \Delta_C u(x,y)=-\text{div}\left[A(x) \nabla u(x,y)\right]\end{equation} 
when $u$ is in the domain of $\Delta_C$.
Given the particular anisotropicity of the operator $\Delta_C$, it is not possible to equip $\R^2$ with a measure $\mu_C$ so that
functions $u$ that solve the equation $(-\Delta_C)^\theta u=f$ are characterized as the minimizers of the Besov 
energy\footnote{In principle it may still be possible to find a symmetric kernel $K(x,y)$ so that  one can obtain solutions of  
the fractional powers of the divergence form operator~\eqref{divform} as minimizers of a non-local energy of type~\eqref{k-energy}.
}
\[
\mathcal{E}_C(u):=\int_{\R^2}\int_{\R^2}\frac{|u(z_1)-u(z_2)|^2}{\mu(B(z_1,|z_1-z_2|))\, |z_1-z_2|^{2\theta}}\, 
d\mathcal{L}^2(z_2)\, d\mathcal{L}^2(z_1).
\]

While fractional powers of divergence form linear operators such as 
$$L_A u= \sum_{i=1}^n \sum_{j=1}^n \partial_{x_i} \bigg( a_{ij}(x) \partial_{x_j} u\bigg),$$
with $L^\infty$, strictly positive definite, symmetric,  coefficients $a_{ij}=a_{ji}$ 
 can be defined and studied using spectral 
decomposition methods as done for instance in \cite{CSt} (see also \cite{StVa} for the non-divergence form analogue), as far as we know our approach is the only 
one in the literature that would allow  to study fractional powers of the analogue non-linear diffusion operator in $\R^n$ of the form

\begin{equation}\label{p-diffusion}
\Delta_p^g u := \frac{1}{\sqrt{g}} \sum_{i,j=1}^{n}\partial_i \bigg[ \sqrt{g} \bigg( \sum_{k,l,m,h=1}^{n}
g_{kl} g^{lm} \partial_m u g^{kh}\partial_h u
\bigg)^{\frac{p-2}{2}} g^{ij} \partial_j u
\bigg],
\end{equation}
where $g=\{ g_{ij}\}$ is an $L^\infty$ symmetric, positive definite matrix valued function, $g^{ij}$ denote the components of its inverse, and $\sqrt{g}$ denotes the determinant of $g$. 
In fact such an operator arises as the $p$-Laplacian associated to the choice of Cheeger differential structure
\[
(\nabla_C u)_i = \sum_{j=1}^ng^{ij} \partial_j u
\]
for $i=1,...,n$.  
For related results in the literature see \cite{GZZ}, for a proof of the existence of minimizers of the 
$p$-Besov energy on Riemannian manifolds, and \cite{CFS} for the study of nonlocal minimal surfaces on 
closed Riemannian manifolds.  

We conclude this section by observing that our approach at the moment does not yet allow to study fractional powers of operators such as \eqref{divform} with non symmetric coefficients. Even in the Eucludean case the current literature has focused mostly on the study of fractional operators  with symmetric kernels, such as those described in \eqref{k-energy}.  A notable exception is the work of   Kassmann and Weidner, see for instance 
\cite{KaWe1, KaWe2, We}. For a different perspective see also  \cite{GKS2}. 
\end{example}

\footnotesize	
\bigskip
	
	\noindent Address:\\
	
	\noindent L.C.: Department of Mathematical Sciences, Smith College,  Northampton, MA 01060, U.S.A.. 
		\noindent E-mail: {\tt lcapogna@smith.edu }\\
	
	\noindent R.G.: Department of Mathematics, Physics and Geology, Cape Breton University, Sydney, NS~B1Y3V3, Canada. 
	\noindent E-mail: {\tt ryan\textunderscore gibara@cbu.ca}\\
	
	\noindent R.K.: Aalto University
Department of Mathematics and Systems Analysis, P.O.~Box~11100,  FI-00076 Aalto,  Finland. 
	\noindent E-mail: {\tt riikka.korte@aalto.fi}\\
	
	\noindent N.S.: Department of Mathematical Sciences, P.O. Box 210025, University of
	Cincinnati, Cincinnati, OH 45221--0025, U.S.A.. 
	\noindent E-mail: {\tt shanmun@uc.edu} 
	
\end{document}